\documentclass[11pt]{preprint}

\usepackage[margin=2.8cm]{geometry}
\usepackage[T1]{fontenc} 
\usepackage[latin1]{inputenc}
\usepackage[english]{babel}
\usepackage[babel]{csquotes}
\usepackage{cite}
\usepackage{amssymb}
\usepackage{amsmath}
\usepackage{amsthm}
\usepackage{latexsym}
\usepackage{graphicx}
\usepackage{mathrsfs}
\usepackage{mathrsfs}
\usepackage{bbm}
\usepackage{verbatim}
\usepackage[colorlinks=true, pdfstartview=FitV, linkcolor=colorLink, citecolor=colorCite, urlcolor=colorLink, linktocpage=true]{hyperref}
\usepackage{enumerate, enumitem}
\usepackage{tabularx}

\numberwithin{equation}{section}

\usepackage[usenames,dvipsnames]{color}

\newtheorem{thm}{Theorem}[section]
\newtheorem{cor}[thm]{Corollary}
\newtheorem{lem}[thm]{Lemma}
\newtheorem{prop}[thm]{Proposition}
\newtheorem{defin}[thm]{Definition}
\newtheorem{remark}[thm]{Remark}

\def\enne{\mathbb{N}}

\def\erre{\mathbb{R}}

\def\P{\mathbb{P}}

\def\E{\mathop{{}\mathbb{E}}}

\def\cL{\mathscr{L}}
\def\cF{\mathscr{F}}
\def\cB{\mathscr{B}}
\def\eps{\varepsilon}

\def\OO{\mathcal{O}}
\renewcommand{\d}{{\mathrm d}}

\def\beq{\begin{equation}}
\def\eeq{\end{equation}}

\def\to{\rightarrow}

\def\embed{\hookrightarrow}

\def\norm #1{\left\|#1\right\|}

\def\sp #1#2{\left<#1,#2\right>}
\newcommand\ip\sp

\allowdisplaybreaks

\definecolor{colorLink}{RGB}{0,100,162}
\definecolor{colorCite}{RGB}{8,124,100}

\setlist[itemize]{leftmargin=5mm}
\setlist[enumerate]{leftmargin=10mm}

\begin{document}

\title{Weak uniqueness by noise for singular stochastic PDEs}

\author{Federico Bertacco$^1$, Carlo Orrieri$^{2}$, Luca Scarpa$^{3}$}

\institute{Imperial College London, Email: \href{mailto:f.bertacco20@imperial.ac.uk}{\color{black} f.bertacco20@imperial.ac.uk} \and Universit\`a di Pavia, Email: \href{mailto: carlo.orrieri@unipv.it}{\color{black} carlo.orrieri@unipv.it} \and Politecnico di Milano, Email: \href{mailto: luca.scarpa@polimi.it}{\color{black} luca.scarpa@polimi.it}}

\maketitle

\begin{abstract}
We prove weak uniqueness of mild solutions for general classes of SPDEs on a Hilbert space. The main novelty is that the drift is only defined on a Sobolev-type subspace and no H\"older-continuity assumptions are required. This framework turns out to be effective to achieve novel uniqueness results for several specific examples. Such wide range of applications is obtained by exploiting either coloured or rougher-than-cylindrical noises.

\vspace{2mm}

\vspace{2mm}
\end{abstract}

\setcounter{tocdepth}{2}
\tableofcontents

\maketitle

\section{Introduction}
\label{sec:intro}
We are interested in weak uniqueness by noise for 
stochastic evolution equations of the form 
\beq
\label{eq0}
  \d X + AX\,\d t = B(X)\,\d t +  G\,\d W\,, \qquad X(0)=x\,,
\eeq
where $A$ is a linear self-adjoint maximal monotone 
operator on a Hilbert space $H$ with an effective domain $D(A)$, 
$W$ is an $H$-cylindrical Wiener process,
and $x\in H$.
As far as the drift is concerned, we are interested
in examples where $B$ is a differential-type operator,
meaning that $B$ could be defined only 
on a Sobolev-type subspace of $H$
and may take values in a dual space larger than $H$. 
More precisely, we suppose that 
\[
  B:D(A^\alpha)\to D(A^{-\beta})
\]
for some constants $\alpha\in[0,1)$ and $\beta\in[0,\frac12)$.
Regarding the covariance operator $G$,
we generally require that it is linear and continuous from $H$ to
$D(A^\delta)$ for some $\smash{\delta\in[0,\frac12)}$ and
nondegenerate, with the main prototype being $G=A^{-\delta}$.
Alternatively, in some specific pathological situations we shall consider $G$ to be 
in the differential form $G=A^{\gamma}$ with $\gamma>0$, 
in order to include a noise which is rougher than the cylindrical one.
Precise assumptions on the setting are 
presented in Section~\ref{sec:main}. Typical motivating examples of SPDEs
that can be covered in this setting are
differential perturbations of the heat equation, Burgers equation,
and Cahn-Hilliard equation, as well as reaction-diffusion equations.

\subsection{Existing literature}
Starting from the pioneering contribution by Veretennikov \cite{V1980}, the literature on uniqueness by noise has witnessed important developments in the infinite-dimensional case in recent years. The first results in this direction were obtained in \cite{GP1993} in the context of quasilinear equations, and in \cite{Z2000, ABGP2006}, where weak uniqueness was studied for equations with nondegenerate multiplicative noise and H\"older-continuous bounded drift. This was then refined in \cite{D2003}, where the drift term has the specific form $\smash{B=A^{\frac12}F}$ for some $F:H\to H$ that is H\"older continuous and bounded with a sufficiently small H\"older norm. The most recent result on weak uniqueness that we are aware of is \cite{priola2, priola2_corr}, in which the particular case $B=A^{\frac12}F$ is considered where $F$ is merely locally H\"older-continuous and the assumptions of smallness and boundedness on $F$ are removed.
Besides weak uniqueness, pathwise uniqueness by noise has also been investigated in the infinite-dimensional setting. In this direction, it is worth mentioning \cite{DPF}, which deals with equations with H\"older continuous drift defined everywhere on $H$, and \cite{Dap1, DPFPR2} for equations with Borel-measurable drift, bounded and unbounded, respectively. 
For completeness, we also refer to the works \cite{AB23, FF2011, FGP2010, KR2005}. Strong uniqueness by noise has also been studied in
 \cite{MP2017, AMP2023} using techniques from backward stochastic equations.

The addition of noise has also been shown to 
be effective in providing refined regularisation phenomena.
In this spirit,
we mention the contributions
\cite{GG2019} on stochastic Hamilton-Jacobi equations,
\cite{GS2017, GM2018} for stochastic scalar conservation laws,
\cite{M2020, BMX2023} on non-blow-up effects for stochastic equations,
\cite{BW2022} on $p$-Laplace equations, and
\cite{GL2023, FL2021} in the context of fluid dynamics.

\subsection{Novelty and main goals}
The above-mentioned results on weak uniqueness by noise for infinite-dimensional equations in the form \eqref{eq0} only deal with the particular case $B:H\to H$, which corresponds to the specific choice $\alpha=\beta=0$. The case $\beta>0$ can be handled only if $\alpha=0$, i.e., when the drift $B$ takes the specific form $A^{\beta}F$, where $F:H\to H$ (see, e.g., \cite{priola2, priola2_corr}). Roughly speaking, if $A$ is a classical second-order differential operator on some domain, the limitation $\alpha=\beta=0$ means that the perturbation $B$ acts as an operator of order zero, e.g., the superposition operator associated with a measurable real function. This excludes the possibility of considering singular perturbation terms in differential form arising in several applications, which are dominated by $A$ itself. For example, this is the case for drifts $B$ in the form $B(\cdot)=F(A^\alpha\cdot)$, which may depend on derivatives of the underlying process in a nonlinear way.

The second main limitation of the available literature consists in the form of the covariance operator of the noise, which is chosen to be the identity in the vast majority. One of the reasons for such a specific setting is due to the formal intuition that ``\emph{the rougher the noise, the more regularising the transition semigroup}'', in the sense that the resolvent operator associated with the Kolmogorov equation of the SPDE is more smoothing when $G$ is close to the identity. However, the choice $G=\operatorname{Id}$ (or equivalently $\delta=0$) comes with an important drawback in terms of the actual solvability of the SPDE
in the mild sense. 
For example, it imposes strong constraints on the spatial dimension of the domain, allowing only dimension one in the case of second-order operators. Moreover, if $G$ is the identity, no additional space regularity can be obtained for the stochastic convolution in general, so one is forced to work on the whole space $H$ by considering the choice $\alpha=0$.
More precisely, if the noise is not coloured (i.e.~$\delta=0$) and $\alpha>0$, 
existence of solutions with values in the relevant space $D(A^\alpha)$ fails in general 
due to the lack of regularity on the stochastic convolution.

At an intuitive level, the main idea is that, on the one hand, if $\delta$ is ``close'' to zero, it is simpler to obtain uniqueness results via the Kolmogorov equation, but the existence of mild solutions for the SPDE requires several restrictions (e.g, $\alpha=0$ and low dimensions). 
On the other hand, if $\delta$ is ``larger'', then one loses some regularisation effects on the Kolmogorov operator but gains in the range of solvable SPDEs that can be considered
(e.g., $\alpha>0$ and higher dimensions).

The main goal of this article is to prove weak uniqueness of mild solutions 
for stochastic equations
in the form \eqref{eq0} with possibly singular drift $B$ (i.e.~$\alpha>0$)
and a possibly coloured noise $G$. In particular, this contribution can be regarded as a first step toward developing a general approach for singular equations. Naturally, such a general framework can be further refined when dealing with specific examples by exploiting the typical features of the equation under consideration, 
such as Burgers-type and fluid-dynamical models.
This could be the subject of future, more specific investigations.

In the present paper, we show that by carefully calibrating the parameter $\delta$, 
it is possible to achieve uniqueness of solutions to the SPDE 
in the more general case of $\alpha,\beta>0$.
This is carried out in a wide generality for local-in-time solutions, 
in order to cover as many examples as possible.
More precisely, the main relation between the parameters involved turns out to be 
\[
  0\leq\delta+\beta<\frac12\,,
\] 
which comes naturally from the regularising properties of the associated 
Ornstein-Uhlenbeck semigroup (see \eqref{SF3} with $\gamma=\beta$).
The restrictions on the parameter $\alpha$ are more subtle 
as they are evident only in the case of unbounded drift in the form 
\[
0\leq\alpha+\beta\leq\frac12
\] 
and are related to maximal regularity arguments 
(see Section~\ref{sec:uniq_unb}). The two limiting cases where 
either $(\alpha,\beta)=(0,\frac12)$ or $(\alpha,\beta)=(\frac12,0)$
will be referred to as {\em critical cases}. The former case 
is not covered here, but
has been investigated with the specific choice $\delta=0$
in \cite{priola2, priola2_corr} under a local H\"older-type condition on
the drift. The latter case is covered instead by the present paper.
Uniqueness for the super-critical case $\alpha=0$ and $\beta>\frac12$ 
is an open problem with the choice of a purely cylindrical noise (i.e.~$\delta=0$),
as pointed out in \cite[Rem.~4]{priola2_corr}. 
Here, we provide a partial answer to this:
precisely, we prove that weak uniqueness holds also in the super-critical case
by employing a rougher-than-cylindrical noise, i.e.~in the 
differential form $A^\gamma \d W$ with $\gamma>0$.

We emphasise that the focus of the paper is {\em not} 
to provide a unifying condition that guarantees existence of solutions, 
rather to establish uniqueness as soon as any local solution exists.
Here, we only provide a sufficient condition for 
global-in-time existence of solutions, 
without exploiting any specific structure of $B$.
For what concerns local/global existence, several sharp results 
are already available in the literature, depending on 
the specific structure of the drift $B$ in terms of, e.g., 
monotonicity \cite{LiuRo}, geometric conditions \cite{GP}.

\subsection{Examples and applications}
We briefly list below some examples of stochastic equations of the form \eqref{eq0} for which we can show weak uniqueness, and we refer to Section~\ref{sec:appl} for further details. 
As far as we know, the uniqueness results that we achieve here
are either completely novel or improvements/extensions
of previous contributions in the literature.

We fix a smooth bounded domain $\OO\subset\erre^d$ ($d\geq1$) and a time horizon $T >0$ so that all the equations considered below are defined in $[0, T] \times \OO$. 
For convenience, in the technical 
Proposition~\ref{prop:lap} we provide a list of the main properties of the 
covariance operator in the form $G=A^{-\delta}$, which can be used in the different settings.

\begin{enumerate}[label=(\alph*)]
	\item  We can deal with \emph{heat equations with polynomial perturbation} in the form 
\begin{equation*}
\d X - \Delta X\,\d t = F(X)\,\d t + G\,\d W \,, \qquad X(0)=x_0 \,,
\end{equation*}
with Dirichlet boundary conditions and where 
$F:\erre\to\erre$ 
is a continuous function that behaves like a polynomial of order $p-1$ with $p\geq2$ and $G\in \cL(L^2(\OO),L^2(\OO))$. 
Here, it is possible to apply our setting under a specific choice of $\alpha$, 
$\beta$ and $\delta$, depending on $p$ and the space dimension.
More precisely, when $p=2$ we have a linearly-growing drift and
in the notation of the paper we work with $\alpha=\beta=0$: in this case,
for initial data $x_0\in L^2(\OO)$
we are able to show uniqueness in law for every $\delta\in[0,\frac12)$ in dimension $d=1$, for every $\delta\in(0,\frac12)$ in dimension $d=2$, and for every $\delta\in(\frac14,\frac12)$ in dimension $d=3$.
This extends both the classical results \cite{DPF, Dap1} 
to higher dimensions $d=2,3$ and the work 
\cite{AMP2023} for what concerns the choice of the drift.
When $p>2$ we have a genuine reaction-diffusion equation with superlinear drift:
in this case, for initial data in suitable subspaces of $L^2(\OO)$ 
(see Section~\ref{sec:appl} for details),
it turns out that 
uniqueness in law holds in dimensions $d=1,2$ for every growth $p\geq2$ 
and in dimension $d=3$ for every growth $p\in[2,4)$. 
These are the first results on uniqueness by noise for general 
reaction-diffusion equations.

\item We can address \emph{heat equations with perturbation in divergence-like form} that can be written as follows
\begin{equation*}
\d X - \Delta X\,\d t = (-\Delta)^\beta F(X)\,\d t + G\,\d W \,, \qquad X(0)=x_0\,,
\end{equation*}
with Dirichlet boundary conditions and where $\beta\in(0,\frac12)$, $F:\erre\to\erre$ is continuous and linearly bounded, $x_0\in L^2(\OO)$, and $G\in \cL(L^2(\OO),L^2(\OO))$. In this case, we have that $\alpha=0$. In dimension $d = 1$, we can prove uniqueness in law for the choice of parameters $\beta\in(0,\frac12)$ and $\delta\in[0,\frac12-\beta)$. 
We can also obtain uniqueness in law in dimensions $d=2$ and $d=3$ for the choices of parameters $\beta\in(0,\frac12)$ and $\delta\in(0,\frac12-\beta)$, and $\beta\in(0,\frac14)$ and $\delta\in(\frac14,\frac12-\beta)$, respectively.
We refer to \cite{priola2, priola2_corr} for the critical case 
$\beta=\frac12$ and $\delta=0$
(only in dimension $d=1$). Furthermore, we are also able to show 
uniqueness in law in the supercritical case $\beta\in[\frac12,\frac34)$
with the choice $G=A^{\gamma}$ with $\gamma\in(0,\frac14)$.

\item We can consider \emph{heat equations with perturbation in non-divergence-like form} that can be written as follows
\begin{equation*}
\d X - \Delta X\,\d t = F((-\Delta)^\alpha X)\,\d t + G\,\d W \,,\qquad X(0)=x_0\,,
\end{equation*}
with Dirichlet boundary conditions and where $\alpha\in(0,1)$, $F:\erre\to\erre$ is continuous and linearly bounded, and $G\in \cL(L^2(\OO),L^2(\OO))$. In this case, we have that $\beta=0$. We can show that for initial data $x_0\in D((-\Delta)^\alpha)$,
uniqueness in law holds 
in dimension $d=1$ for the choice of parameters $\alpha\in(0,\frac12]$ and 
$\delta\in[0,\frac12)$ 
(with also $\delta\in(\frac14,\frac12)$ in the critical case $\alpha=\frac12$),
in dimension $d=2$ with the choices $\alpha\in(0,\frac12)$ and 
$\delta\in(\alpha,\frac12)$, and 
in dimension $d=3$ with the choices $\alpha\in(0,\frac14)$ and 
$\delta\in(\alpha+\frac14,\frac12)$.
Moreover, if $F$ is also bounded, then uniqueness is law holds 
for every $\alpha\in(0,1)$ and $\delta\in[0,\frac12)$ in every dimension $d=1,2,3$
and for initial data $x_0\in L^2(\OO)$.
Up to our knowledge, this is one of the first contributions in the literature
dealing with a fully-nonlinear drift.

\item We can deal with \emph{one-dimensional Burgers equations with perturbation} 
in the form
\begin{equation*}
\d X - \Delta X\,\d t = X\cdot\nabla X\,\d t + F(X,\nabla X)\,\d t
+ G\,\d W \,, \qquad X(0)=x_0 \,,
\end{equation*}
with Dirichlet boundary conditions, where $F:\erre^2\to\erre$ is measurable and bounded,
$x_0\in H^1_0(\OO)$ and $G\in \cL^2(L^2(\OO),L^2(\OO))$. Thanks to our results, we can prove that uniqueness in distribution holds for $\alpha=\frac12$, $\beta=0$, and 
$\delta\in(\frac14,\frac12)$. This uniqueness result comes as a byproduct of our strategy,
even though it is already available in the literature via other techniques (see \cite{LiuRo}).

\item Finally, we can also consider \emph{Cahn--Hilliard equations with perturbation} in the form
  \begin{equation*}
  \d X -\Delta (-\Delta X +F_1(X))\,\d t 
  = F_2(X, \nabla X, D^2X)\,\d t + G\,\d W \,, \qquad
  X(0)=x_0\,,
  \end{equation*}
with Neumann-type boundary conditions for $X$ and $\Delta X$, and where $F_1:\erre\to\erre$ is the derivative of the classical polynomial double-well potential as defined in \eqref{eq:doubleWellPot}, $F_2:\erre\times\erre^d\times\erre^{d\times d}\to\erre$ is continuous and linearly bounded in every argument, and $G\in \cL(L^2(\OO),L^2(\OO))$. In the notation of the paper, we have that $\alpha=\frac12$ and $\beta=0$. We can treat dimensions $d=1,2,3$ and we can shows that uniqueness in law holds for every $\delta\in(\frac{d}8,\frac12)$
and for initial data $x_0\in H^2(\OO)$ with $\partial_{\bf n}x_0=0$ on $\partial\OO$.
Again, this is the first contribution on uniqueness by noise for 
a Cahn-Hilliard equation with singular source possibly depending on 
higher space derivatives of the underlying process.
We mention \cite{priola2, priola2_corr} for the case of linearly bounded potential 
and non-singular perturbation.
\end{enumerate}

\subsection{Strategy and technical issues}
Let us begin by emphasising the main technical assumptions on the operators. The only technical requiremens on $B$ are a local boundedness condition from
$D(A^\alpha)$ to $D(A^{-\beta})$
as well as a weak continuity assumption of the form
\begin{equation*}
B:D(A^\alpha)\to Z \quad\text{is strongly-weakly continuous,}
\end{equation*}
where $Z$ is a Hilbert space containing $D(A^{-\beta})$. Since $Z$ can be chosen to be arbitrarily large, and continuity is intended only in a weak sense on $Z$, this condition can be easily checked in practice according to the specific example under consideration. In particular, we stress that \emph{no} H\"older-continuity assumption is required for $B$. This contrasts the above-mentioned literature where this assumption is often employed in order to exploit classical Schauder-type estimates. 
Furthermore, we stress that in the case where $B$ is also bounded, 
the above continuity requirement is not needed, and only measurability of $B$ suffices.

The only technical assumption on the covariance operator $G$ 
is a compatibility condition with the semigroup generated by $-A$, 
as proposed by Da Prato and Zabczyk in \cite[Prop.~6.4.2]{dapratozab2}.
This is automatically satisfied in the Hilbert-Schmidt case, but it is 
significantly weaker.
The only case where we need a Hilbert-Schmidt requirement on $G$
is the critical case $\alpha=\frac12$ and $\beta=0$.

In terms of the strategy, roughly speaking, the idea is to write an It\^o formula for $u(X)$, where $X$ is a suitable weak solution to \eqref{eq0}, and $u$ is a suitable solution to the associated Kolmogorov equation given by:
\beq\label{eq1}
  \lambda u(x)
  -\frac12\operatorname{Tr}(QD^2u(x)) + (Ax, Du(x))
  = f(x) + \langle B(x),Du(x)\rangle\,, \quad x\in D(A)\,,
\eeq
where $Q:=G^*G$, $f$ is a given bounded continuous function on $H$, $\lambda>0$ is a fixed coefficient, and $\langle\cdot,\cdot\rangle$ denotes the duality pairing between $D(A^{-\beta})$ and $D(A^\beta)$.
Proceeding formally, if one is able to give sense to the It\^o formula for $u(X)$, then direct computations show that
\beq\label{eq2}
  u(x)=\E\int_0^{+\infty}e^{-\lambda t}f(X(t))\,\d t\,.
\eeq
Since this holds for every weak solution $X$ starting from $x$, we can infer the uniqueness of the law of $X$ in the space of continuous trajectories with values in $H$.

In order to rigorously justify the aforementioned procedure, several technical issues need to be addressed. Firstly, it should be noted that $X$ is generally only a mild solution to equation \eqref{eq0}, and direct application of the It\^o formula to $X$ is not possible. Therefore, a suitable regularisation $(X_j:=P_jX)_j$ must be introduced, where $(P_j)_j$ is a family of projections converging to the identity as $j\to\infty$. 
Secondly, a similar issue arises with the Kolmogorov equation \eqref{eq1}, as one is only able to show existence of a suitable mild solution $u$ that is only of class $C^1$ on $H$. Consequently, direct application of the It\^o formula to $u$ is not possible. Therefore, it becomes necessary to introduce an additional sequence of regularized functions $(u_\eps)_\eps$. These functions are solutions to the Kolmogorov equation, where the original functions $f$ and $B$ are replaced by smoothed counterparts $(f_\eps)_\eps$ and $(B_\eps)_\eps$ obtained through convolution with Gaussian measures on $H$.
Eventually, one is able to justify rigorously the It\^o formula 
for $u_\eps(X_j)$, which reads, after taking expectations 
and doing standard computations, as follows
\begin{align*}
  &\E\left[u_\eps(X_j(t))\right] - u_\eps(x_j) -\lambda\E\int_0^tu_\eps(X_j(s))\,\d s
  +\E\int_0^tf_\eps(X_j(s))\,\d s\\
  &= \underbrace{\E\int_0^t\left(P_jB(X(s))-B_\eps(X_j(s)), 
  Du_\eps(X_j(s))\right)\,\d s}_{I_1(\eps,j)}
  +\underbrace{\frac12\E\int_0^t\operatorname{Tr}\left[(P_j-I)QD^2
  u_\eps(X_j(s))\right]\,\d s}_{I_2(\eps,j)}\,,
\end{align*}
where $I_1$ and $I_2$ have to be interpreted as remainders 
of the limiting balance yielding \eqref{eq2}.
The novel technical idea is to pass to the limit jointly in $j$ and $\eps$
by using a suitable scaling of the parameters in such a way that 
$I_1(\eps,j), I_2(\eps,j)\to 0$.

In order to pass to the limit in the It\^o formula, 
for the terms on the left-hand side we 
need some convergence of $X_j\to X$
(which easily follows by the choice of $(P_j)_j$) 
and a convergence $u_\eps\to u$ at least in
some H\"older space $C^{0,\gamma}_b(H)$ for some $\gamma\in(0,1]$.
While for the term $I_1$ we only need a uniform bound on the gradients 
$(Du_\eps)_\eps$ at least in $C^0_b(H;D(A^\beta))$,
the term $I_2$ is exactly the one that can be treated 
by choosing a suitable scaling of $\eps$ and $j$. Indeed, 
roughly speaking, the second derivatives $D^2u_\eps$ 
blow up as $\eps\to0$ while $P_j-I$ vanishes as $j\to \infty$.
As a by product of this joint limiting procedure, 
we are able to avoid any H\"older regularity on $B$, and what really matters
is that $B_\eps$ is H\"older for every $\eps$.

The proof of convergence of $u_\eps$ to $u$ and $Du_\eps$ to $Du$
is based on compactness and density arguments in infinite dimension.
More precisely, by using the regularising properties of the Kolmogorov operator,
we first show uniform bounds of $(u_\eps)_\eps$ in $C^1_b(H)$
and of $(Du_\eps)_\eps$ in $C^0_b(H;D(A^\beta))$.
By the Arzel\`a--Ascoli theorem and a localisation procedure on $H$,
this allows to infer the convergence $u_\eps\to \mathfrak u$
and $Du_\eps\to D\mathfrak u$ pointwise in $H$
for some candidate $\mathfrak u\in C^1_b(H)$.
Therefore, the problem is reduced to the identification $u=\mathfrak u$.
Indeed, since in general such convergences are only pointwise in $H$,
and not with respect to the uniform topology,
is it not possible to pass to the limit in the classical $C^1$-sense
in the Kolmogorov equation. 
Nonetheless, 
the convergence in the Kolmogorov equation can be proved in
$L^2(H, N_{Q_\infty})$, where $N_{Q_\infty}$ is the nondegenerate 
Gaussian invariant measure of the associated Ornstein-Uhlenbeck semigroup.
By uniqueness of the limiting Kolmogorov equation in $L^2(H, N_{Q_\infty})$,
it follows that $u=\mathfrak u$, $N_{Q_\infty}$-almost-everywhere on $H$.
At this point, we exploit a lemma originally due It\^o (see \cite{ito}). In particular,
this guarantees that, since $N_{Q_\infty}$ is nondegenerate,
every Borel subset $E\subset H$ with full measure $N_{Q_\infty}(E)=1$
is dense in $H$. By Lipschitz-continuity of $u$ and $\mathfrak u$, this 
allows us to identify $u=\mathfrak u$ on the whole $H$.

Let us stress that the convergence of $I_2$ is not trivial since
$Q$ is not of trace-class. 
However, it is possible to exploit the regularising 
properties of the Kolmogorov operator to show that 
$QD^2u_\eps$ is continuous with values in the trace-class operators,
even if $Q$ is not trace-class itself. 
This allows us to deduce the convergence of $I_2$ to $0$ anyway.

All the arguments described above are done in detail for the case
where $B$ is bounded from $D(A^\alpha)$ to $D(A^{-\beta})$.
For the unbounded case, we use a localisation-extension procedure on
the solutions to the SPDEs and exploit the available uniqueness result 
for the bounded case. More precisely, given a solution $X$ to \eqref{eq0}
in the unbounded case, we first employ some stochastic maximal regularity 
argument to show that $X$ is actually continuous with values in $D(A^\alpha)$. This can be done thanks to the results contained in 
\cite{AV2022, vNVW}. 
Since $B$ is only defined in $D(A^\alpha)$, it is necessary to introduce a sequence of stopping times $(\tau_N)_N$, such that the
stopped processes 
$(X^{\tau_N})_N$ solve the 
equation \eqref{eq0} locally on $[0,\tau_N]$ 
with $B$ replaced by some truncation $B_N$.
Eventually, we show that $X^{\tau_N}$ can be extended 
globally on $[0,+\infty)$ to a solution $\tilde X_N$ of \eqref{eq0}
with $B_N$. Since $B_N$ is bounded, the available uniqueness 
result in the bounded case yields, together with the arbitrariness of $N$,
uniqueness in law also for $X$.

\subsection{Structure of the paper}
The remaining part of this article is organised as follows.
In Section~\ref{sec:main}, we introduce the general setting of the paper
and list the main uniqueness results. More precisely, Theorems~\ref{thm2}-\ref{thm3} contain the results on weak uniqueness in the bounded and unbounded cases.
In Section~\ref{sec:kolm}, we study the limit Kolmogorov equation 
in a mild sense and its approximation through suitably regularised equations.
The proofs of the main results are given in Section~\ref{sec:uniq},
while applications to specific example of interest are thoroughly discussed
in Section~\ref{sec:appl}. 
Finally, in Appendix~\ref{sec:ex} in the bounded case
we provide a sufficient condition for existence
of global weak solutions to equation \eqref{eq0}, which 
is used to extend local solutions in the uniqueness argument mentioned above.

\section{Precise setting and main results}
\label{sec:main}
In this section, we precisely state our assumptions and the main results of this article. 
Throughout the paper, we let $H$ be a separable Hilbert space 
with scalar product $(\cdot, \cdot)$ and norm $\|\cdot\|$,
which will always be identified to its dual space $H^*$. We denote by $\cL^1(H,H)$ and $\cL^2(H,H)$ the usual 
spaces of trace-class and Hilbert-Schmidt operators on $H$, respectively.
Moreover, we also introduce the space
\[
  \cL^4(H,H):=\left\{L\in\cL(H,H): \sum_{k=0}^\infty
  \norm{Le_k}_H^4<+\infty
  \right\}
\] 
where $(e_k)_k$ is a complete orthonormal system of $H$, endowed with 
obvious choice of the norm. Eventually, for $k\in\enne$,
we use the symbols $C^k_b(H)$ and $UC_b(H)$ for the spaces of
$k$-times bounded Fr\'echet-differentiable functions on $H$ with continuous derivatives
and uniformly continuous bounded functions on $H$.
Analogously, we employ the notation $C^{k,\ell}_b(H)$ and $C^{k+\ell}_b(H)$
for the space of $k$-times bounded Fr\'echet-differentiable functions on $H$ with
$\ell$-H\"older derivatives.

We assume the following condition.
\begin{enumerate}[start=1,label={{(H\arabic*})}]
  \item \label{H1} (\textbf{Assumptions on} $A$) The operator $A$ is a linear self-adjoint maximal monotone operator
  with compact resolvent. 
  In particular,
  $(-A)$ generates a strongly continuous symmetric 
  semigroup $(S(t))_{t\geq0}$ of contractions 
  on $H$.
\end{enumerate}

Before stating the remaining conditions, we note that by assumption \ref{H1}, the fractional powers $A^{s}$, 
for $s\in(0,1)$,
are well-defined. Indeed, if $(e_k)_{k\in\enne}$ is a complete orthonormal system of $H$
made of eigenvectors of $A$ and
$(\lambda_k)_{k\in\enne}$ are the corresponding eigenvalues, one can set 
\begin{align*}
  D(A^s)&:=\left\{x\in H: \; \sum_{k=0}^\infty\lambda^{2s}_k (x,e_k)^2<+\infty\right\}\,, \qquad A^{s}x :=\sum_{k=0}^\infty\lambda^s_k (x,e_k)e_k\,, \quad x\in D(A^s)\,.
\end{align*}
For brevity of notation, we define 
\beq
  \label{V_s}
  V_{2s}:= D(A^s)\,, \quad s\in(0,1]\,,
\eeq
and we recall that  $V_{2s}$ is a separable Hilbert space with scalar 
product and norm given by
\begin{align*}
  (x,y)_{2s}&:=(A^sx, A^sy)\,, \quad \forall x,y\in V_{2s}\,, \qquad \norm{x}_{2s}:=\norm{A^sx}\,, \quad \forall x\in V_{2s}\,,
\end{align*}
respectively. The duality pairing between $D(A^{-s})$ and $D(A^s)$
will be denoted by the symbol $\langle\cdot,\cdot\rangle$,
without explicit notation for $s$.
Moreover, the inclusion $V_{2s}\embed H$  is dense for every $s\in(0,1]$
and by definition it holds that 
\[
  (Ax,x)=\|A^{1/2}x\|^2=\norm{x}^2_{1}\,, \quad \forall x\in D(A)\,.
\]
In particular, the semigroup $S$ is also analytic and of negative type (see \cite[Prop.~A12]{dapratozab}). Furthermore, 
for every $s,s'>0$ with $s<s'$ the inclusion $V_{2s'}\embed V_{2s}$ is compact. 

For all $s\in(0,1]$, we naturally set $D(A^{-s}):=D(A^s)^*$, 
endowed with its natural norm. Note that
the semigroup $S$ extends to a semigroup on $D(A^{-s})$, for 
all $s\in(0,1]$, which will be denoted by the same symbol $S$ for brevity. This follows by extrapolation from the fact that 
the restriction of $S$ to $D(A^s)$ is a strongly continuous 
semigroup of contractions on $D(A^s)$.
Indeed, for all $x\in H$, $y\in D(A^s)$, and $t\geq0$, we have
\begin{equation*}
(S(t)x,y)_H=(x,S(t)y)_H\leq 
\|x\|_{D(A^{-s})}\|S(t)y\|_{D(A^s)}\leq 
\|x\|_{D(A^{-s})}\|y\|_{D(A^s)} \,, 
\end{equation*}
and so the (unique) extension follows by density of $H$ in $D(A^{-s})$.

We can now proceed with the remaining assumptions.

\begin{enumerate}[start=2,label={{(H\arabic*})}]
  \item \label{H2} (\textbf{Assumptions on} $B$) Let $\alpha\in[0,1)$, $\beta\in[0,\frac12)$, and let
  $B:D(A^\alpha)\to D(A^{-\beta})$ be measurable 
  and locally bounded. 
  \vspace*{2mm}
  \item \label{H3} (\textbf{Assumptions on} $G$) 
  Let $\delta\in[0,\frac12-\beta)$ and
  let $G\in\cL(H,H)$ be positive self-adjoint such that
  $G$ commutes with $A$, 
  $\operatorname{ker}(G)=\{0\}$, and
  $G(H)=D(A^{\delta})$.
  We set $Q:=GG^*$ and
 \beq
  \label{Qt}
  Q_t:=\int_0^t S(s)QS(s)\,\d s=
  \frac12 A^{-1}Q(I-S(2t))\,, \quad t\geq0\,.
  \eeq
 We assume that 
 there exist $\xi\in(0,\frac12)$, $\lambda>0$, and $\vartheta\in(0,1)$ such that 
 \begin{align}
 \label{eq:cont_time}
 \int_0^Tt^{-2\xi}\|S(t)G\|^2_{\cL^2(H,H)}\,\d t <+\infty \quad\forall\,T>0\,,\\
 \label{eq:L4}
 \int_0^{+\infty} e^{-\lambda t}
 \norm{Q_t^{-\frac12}S(t)G}_{\cL^4(H,H)}^{2(1-\vartheta)}
 \norm{Q_t^{-\frac12}S(2t)Q}^{\vartheta}_{\cL^2(H,H)}\,\d t<+\infty\,.
 \end{align}
\end{enumerate}

The prototypical choice 
of $G$ satisfying assumption \ref{H3}
is given by $G=A^{-\delta}$, with suitable conditions on $\delta$
(see Section~\ref{sec:appl} for some examples).

Note that condition \eqref{eq:cont_time} is classical and guarantees 
both that $Q_t\in \cL^1(H,H)$ for all $t\geq0$ and
the continuity in time with values in $H$ 
for the stochastic convolution (see \cite[Thm.~5.11]{dapratozab}). 
Moreover, recalling that in our case
the semigroup $S$ is analytic, condition \eqref{eq:cont_time}
implies (see \cite[Thm.~5.15]{dapratozab}) that the stochastic convolution also verifies
\beq\label{eq:cont_conv}
  \int_0^\cdot S(\cdot - s)G\,\d W(s)\in 
  C^0(\erre_+; D(A^{\zeta})) \quad\forall\,\zeta\in[0,\xi)\,, \quad\P\text{-a.s.}
  \eeq
Furthermore, although condition \eqref{eq:L4} may look convoluted at first sight,
it is satisfied in numerous well-known examples 
of stochastic equations in space dimension up to three (this will be shown in detail in Section~\ref{sec:appl}).
The reason why we require \eqref{eq:L4} is actually 
very natural, as it guarantees that strong solutions $u$
to the Kolmogorov equation associated to the 
corresponding Ornstein-Uhlenbeck semigroup 
verify $QD^2u\in\cL^1(H,H)$, even when $Q\notin\cL^1(H,H)$.
This is the only sufficient condition for such regularity 
that we are aware of. We refer to \cite[Prop.~6.4.2]{dapratozab2}
for more detail. 

\begin{remark}
  \label{rmk:hs}
  We point out that if $G\in\cL^2(H,H)$, then conditions \eqref{eq:cont_time}
  and \eqref{eq:L4} are satisfied for every $\xi\in(0,\frac12)$ and 
  $\vartheta\in(\frac{2\delta}{2\delta+1}, 1)$. For a detailed proof see Lemma~\ref{lem:hs}
  in Section~\ref{sec:appl}.
\end{remark}


We now state precisely what we mean by weak solution to problem~\eqref{eq0}.
\begin{defin}[Global weak solution]
  \label{def_sol}
  Assume \ref{H1}-\ref{H2}-\ref{H3} and let $x\in H$. 
  A global weak solution to \eqref{eq0} is a sextuplet 
  \[
  \left(\Omega, \cF, (\cF_t)_{t\geq0}, \P, W, X\right)
  \]
  where $(\Omega, \cF, (\cF_t)_{t\geq0}, \P)$ is a filtered probability space
  satisfying the usual conditions, $W$ is a $H$-cylindrical Wiener process, 
  and $X$ is a progressively measurable process such that
  \begin{align*}
  X  &\in C^0(\erre_+; H)\cap 
  L^1_{loc}(\erre_+; D(A^{\alpha})) 
  \quad \text{ and } \quad 
  B(X) \in L^2_{loc}(\erre_+; D(A^{-\beta}))\,, \quad\P\text{-a.s.}\,,
  \end{align*}
satisfying the following equation
  \beq
  \label{eq_sol}
  X(t) = S(t)x + \int_0^tS(t-s)B(X(s))\,\d s + 
  \int_0^tS(t-s)G\,\d W(s)\,, \quad\forall\,t \geq0\,, \quad\P\text{-a.s.}
  \eeq
\end{defin}

\begin{defin}[Local weak solution]
  \label{def_sol_loc}
  Assume \ref{H1}-\ref{H2}-\ref{H3} and let $x\in H$. 
  A local weak solution to \eqref{eq0} is a septuplet 
  \[
  \left(\Omega, \cF, (\cF_t)_{t\geq0}, \P, W, X,\tau\right)
  \]
  where $(\Omega, \cF, (\cF_t)_{t\geq0}, \P)$ is a filtered probability space
  satisfying the usual conditions, $W$ is a $H$-cylindrical Wiener process, 
  $X$ is a progressively measurable process, and $\tau:\Omega\to(0,+\infty]$ is a
  positive stopping time such that the stopped process $X^\tau:=X(\cdot\wedge\tau)$ satisfies
  \begin{align*}
  X^\tau  &\in C^0(\erre_+; H)\cap 
  L^1_{loc}(\erre_+; D(A^{\alpha})) 
  \quad \text{ and } \quad 
  B(X^\tau) \in L^2_{loc}(\erre_+; D(A^{-\beta}))\,, \quad\P\text{-a.s.}\,,
  \end{align*}
 and
  \begin{align}
  \nonumber
  X^\tau(t) &= S(t\wedge\tau)x + 
  \int_0^{t\wedge\tau}S((t\wedge\tau)-s)B(X^\tau(s))\,\d s \\
  &+ 
  \int_0^{t\wedge\tau}S((t\wedge\tau)-s)G\,\d W(s)\,, \quad\forall\,t \geq0\,, \quad\P\text{-a.s.}
  \label{eq_sol_loc}
  \end{align}
\end{defin}
Note that if $\tau=+\infty$ then the Definitions~\ref{def_sol} and \ref{def_sol_loc} coincide.

\begin{remark} 
\label{rmk:maximal}
We recall that in Definition~\ref{def_sol}, we required that 
  $B(X) \in L^2_{loc}(\erre_+;D(A^{-\beta}))$. 
  However, in order to give sense to \eqref{eq_sol}, one could only
  require that $B(X) \in L^1_{loc}(\erre_+;D(A^{-\beta}))$.
  However, the stronger condition $B(X) \in L^2_{loc}(\erre_+;D(A^{-\beta}))$
  is actually natural as it is always satisfied a posteriori
  by the assumptions made on $B$. Indeed, 
  in the bounded case it trivially holds that $B(X) \in L^\infty(\erre_+;D(A^{-\beta}))$ 
  as soon as $X$ takes values in $D(A^{\alpha})$ almost everywhere.
  In the unbounded case, the requirement $B(X) \in L^2_{loc}(\erre_+;D(A^{-\beta}))$
  is needed in order to exploit a maximal regularity argument that guarantees 
  \beq\label{eq:max_reg}
  \int_0^\cdot S(\cdot - s)B(X(s))\,\d s\in 
  C^0(\erre_+; D(A^{\frac12-\beta})) \,, \quad\P\text{-a.s.}
  \eeq
  Nonetheless, 
  since a posteriori the process $X$ is in $L^\infty(0,T; D(A^{\alpha}))$, the locally 
  boundedness assumption on $B$ guarantees 
  that actually $B(X) \in L^\infty(\erre_+;D(A^{-\beta}))$ 
  also in this case. 
\end{remark}

For completeness, we specify in the following definition what it means for two continuous processes taking values in $H$ to have the same law. 

\begin{defin}[Uniqueness in law]
  \label{def_uniq_law}
Assume that $X$ and $Y$ are two progressively measurable processes defined on a filtered probability space $(\Omega, \cF, (\cF_t)_{t\geq0}, \P)$ such that $X$, $Y \in C^0(\erre_+; H)$, $\P$-a.s. We say that $X$ and $Y$ have the same law on $C^0(\erre_+; H)$ if 
  \[
    \E\left[\psi(X)\right] = \E\left[\psi(Y)\right]\,,
    \quad  \forall \, \psi: C^0(\erre_+; H)\to \erre \text{ continuous and bounded}\,.
  \] 
\end{defin}

We are now ready to state our main results.
The first result is a weak uniqueness property in the case of bounded drift
for global weak solutions: considering global solutions here
is very natural since the boundedness of the drift
guarantees global existence under very reasonable assumptions (see Appendix~\ref{sec:ex}). 

\begin{thm}[$B$ bounded]
  \label{thm2}
  Assume \ref{H1}-\ref{H2}-\ref{H3}, and that 
  $B:D(A^\alpha)\to D(A^{-\beta})$ 
  is measurable and bounded.
  Let $x\in H$ and 
  let 
  \[
  \left(\Omega, \cF, (\cF_t)_{t\geq0}, \P, W, X\right)\,, \qquad
  \left(\Omega, \cF, (\cF_t)_{t\geq0}, \P, W, Y\right)
  \]
  be two global weak solutions to \eqref{eq0} in the sense of Definition~\ref{def_sol}
  with respect to the same initial datum $x$.
  Then, $X$ and $Y$ have the same law in the sense of Definition~\ref{def_uniq_law}.
  \end{thm}

Our second result provides weak uniqueness for local solutions 
in the case of unbounded drift. Working with local solutions in this case is
very natural since global existence may fail in general if the drift is unbounded.
Here the unboundedness of the drift requires a restriction on the parameters
and a weak continuity requirement on $B$.

\begin{thm}[$B$ unbounded]
  \label{thm3}
  Assume \ref{H1}-\ref{H2}-\ref{H3} and let $\alpha\in[0,\frac12-\beta]\cap[0,\xi)$.
  Suppose also that there exists a Hilbert space $Z$ such that $D(A^{-\beta})\embed Z$ 
  and that $B$ is strongly-weakly continuous from $D(A^\alpha)$ to $Z$.
  Let $x\in D(A^\alpha)$ and  
  \[
  \left(\Omega, \cF, (\cF_t)_{t\geq0}, \P, W, X,\tau_X\right)\,, \qquad
  \left(\Omega, \cF, (\cF_t)_{t\geq0}, \P, W, Y,\tau_Y\right)
  \]
  be two local weak solutions to \eqref{eq0} in the sense of Definition~\ref{def_sol_loc}
  with respect to the same initial datum $x$.
  Then, setting $\tau:=\tau_X\wedge\tau_Y$, $X^\tau$ and $Y^\tau$
  have the same law in the sense of Definition~\ref{def_uniq_law}.
  Furthermore, under the stronger condition that $G\in \cL^2(H, D(A^{\delta'}))$
  for some $\delta'>0$,
  then the same conclusion holds also in the limiting case 
  $\alpha=\frac12$ and $\beta=0$.
\end{thm}

\begin{remark}
  Let us point out that the choice of $\alpha$ and $\beta$
  is not intrinsic in $B$. Indeed, 
  it is possible that 
  the same differential operator $B$ satisfies assumption \ref{H2} 
  for different choices of $\alpha$ and $\beta$,
  depending on whether one considers 
  strong or weak formulations of it. 
  What is actually intrinsic in the choice of $B$ is the amplitude 
  $\alpha+\beta$, which represents the differential order of $B$.
  More precisely, given a differential operator $B$ of order $q$,
  one can frame the equation in several ways 
  by choosing $\alpha,\beta$ such that $\alpha+\beta=q$.
  This will be shown in detail in Section~\ref{sec:appl}.
\end{remark}

\begin{remark}
  Note that in Theorem~\ref{thm3}, the requirement $\alpha\in[0,\frac12-\beta]\cap[0,\xi)$
  excludes a priori the limiting case $\alpha=\frac12$ and $\beta=0$.
  However, this issue can be overcome by enforcing the regularity of $G$
  in a Hilbert-Schmidt sense
  as specified in the theorem. This allows us to employ stochastic maximal regularity 
  arguments in the localisation procedure.
\end{remark}

Note that the choice $\tau_X=\tau_Y=+\infty$ in Theorem~\ref{thm3} is allowed 
and corresponds to weak uniqueness for global solutions.
Moreover, let us point out the particular case where
$\alpha=\beta=\delta=0$ is the setting considered in \cite{Dap1, DPFPR2}, while
in the framework of Theorems~\ref{thm2}--\ref{thm3}
when $\alpha=0$ we are also able to cover the case $\beta\in[0,\frac12)$. 
The critical case $\alpha=\delta=0$ and $\beta=\frac12$ has been analysed 
in \cite{priola2} for a specific form of $B$ satisfying some local H\"older condition.

The super-critical case $\beta\in(\frac12,1)$,
as well as the drop of the H\"older assumption 
in the case $(\alpha,\beta)=(0,\frac12)$,
is an open problem when one considers $\delta=0$
as pointed out in \cite[Rmk.~4]{priola2_corr}.
In this direction, we give a partial answer to this question. More precisely,
we show that uniqueness can be achieved also in the 
regime $\beta>\frac12$ if one suitably considers a rougher noise, i.e.
\beq\label{eq:-gamma}
  \d X + AX\,\d t = B(X)\,\d t + A^\gamma\d W
\eeq
with $\gamma>0$.
This is the content of our last result, which follows from 
Theorems~\ref{thm2}--\ref{thm3} by suitably changing the functional setting.

\begin{cor}
  \label{thm4}
  Let $\alpha\in[0,\xi)$, $\beta\in[0,\frac12 + \xi-\alpha)$, and
  $\gamma\in[0,\beta]\cap(\beta-\frac12,\xi-\alpha)$.
  Let the operator $A$ satisfy \ref{H1} and \ref{H3} with the choice $\delta=0$, and
  let $B:D(A^\alpha)\to D(A^{-\beta})$ be measurable and bounded.
  Then, for every $x\in D(A^{-\gamma})$ there exists a sextuplet 
  $(\Omega, \cF, (\cF_t)_{t\geq0}, \P, W, X)$
  where $(\Omega, \cF, (\cF_t)_{t\geq0}, \P)$ is a filtered probability space
  satisfying the usual conditions, $W$ is a $H$-cylindrical Wiener process, 
  and $X$ is a progressively measurable process 
  \[
  X\in C^0(\erre_+; D(A^{-\gamma}))\cap 
  L^1_{loc}(\erre_+; D(A^\alpha))\,, \quad\P\text{-a.s.}\,,
  \]
  with 
    \[
    \int_0^\cdot S(t-s)A^\gamma\,\d W(s) \in C^0(\erre_+; D(A^{-\gamma}))\,,
  \]
  satisfying the following equation
  \[
  X(t) = S(t)x + \int_0^tS(t-s)B(X(s))\,\d s + 
  \int_0^tS(t-s)A^\gamma\,\d W(s)\,, \quad\forall\,t \geq0\,, \quad\P\text{-a.s.}
  \]
  Moreover, the law of $X$ on $C^0(\erre_+; D(A^{-\gamma}))$ is unique.
\end{cor}
A typical application of Corollary~\ref{thm4} is given by divergence-like perturbation 
of the heat equation in a super-critical regime. A detailed example is presented in Section~\ref{sec:appl}.

\section{The Kolmogorov equation}
\label{sec:kolm}

In this section, we assume that $B:D(A^\alpha)\to D(A^{-\beta})$ is bounded,
and we work under assumptions \ref{H1}-\ref{H2}-\ref{H3}.
This will be enough to prove uniqueness in law also for more general operators $B$
by using a localisation argument at the SPDE-level 
(see \cite{DPFPR2} for a similar technique).

We consider the Kolmogorov equation associated to \eqref{eq0}, i.e.,
\beq
 \label{eq_kolm}
 \lambda u(x) + Lu(x) = f(x) + \sp{B(x)}{Du(x)}\,, \quad x\in D(A)\,, 
\eeq
where $\lambda >0$, $f\in UC^0_b(H)$, and the operator $L$ is defined as
\[
  Lv(x):=-\frac12\operatorname{Tr}(QD^2v(x)) + (Ax, Dv(x))\,, 
  \quad x\in D(A)\,,\quad v\in C^2_b(H)\,.
\]
We introduce the Ornstein-Uhlenbeck semigroup associated to $L$ by letting
\[
  [R_tv](x):=\int_Hv(y+S(t)x)\,N_{Q_t}(\d y)\,, 
  \quad x\in H\,, \quad v\in\cB_b(H)\,, \quad t\geq0\,,
\]
where $N_{Q_t}$ denotes the Gaussian measure with mean $0$ and covariance operator $Q_t$ defined in \eqref{Qt}. We recall that the probability measure $N_{Q_\infty}$, with
\beq
  \label{Q_inf}
  Q_\infty:=\int_0^\infty S(s)QS(s)\,\d s=\frac12A^{-1}Q\,,
\eeq
is the unique invariant measure for the semigroup $R$, and that $R$ uniquely extends to
a semigroup of contractions also on $L^2(H,N_{Q_\infty})$.

We collect here some useful properties of the operator $Q_t$.
\begin{lem}\label{lem_strong}
  Assume \ref{H1}-\ref{H2}-\ref{H3}. Then, it holds that 
  \beq
  \label{strong1}
  S(t)(H)\subset Q_t^{1/2}(H)\,, \quad\forall\,t>0\,,
\eeq
and for all $\gamma\geq0$
there exists a constant $C_\gamma>0$, depending on $G$, such that
  \beq
  \label{strong2}
  \norm{A^\gamma Q_t^{-1/2}S(t)}_{\cL(H,H)}\leq 
  \frac{C_\gamma}{t^{\frac12+\delta+\gamma}}\,,
  \quad\forall\,t>0\,.
\eeq
\end{lem}
\begin{proof}
Noting that $Q_t=\frac12 A^{-1}Q(I-S(2t))$ and $G=Q^{1/2}$, we have 
\begin{align*}
  A^\gamma Q_t^{-1/2}S(t)e_k=e^{-t\lambda_k}Q_t^{-1/2}A^\gamma e_k
  =e^{-t\lambda_k}\sqrt2 \lambda_k^{\frac12+\gamma} 
  \frac{1}{\sqrt{1-e^{-2\lambda_k t}}}G^{-1}e_k\,.
\end{align*}
Since $G^{-1}:V_{2\delta}\to H$ is an isomorphism, we deduce that
\begin{align*}
  \norm{A^\gamma Q_t^{-1/2}S(t)e_k}&\leq\norm{G^{-1}}_{\cL(V_{2\delta,H})}
  e^{-t\lambda_k}\sqrt2 \lambda_k^{\frac12+\gamma} 
  \frac{1}{\sqrt{1-e^{-2\lambda_k t}}}\norm{e_k}_{2\delta}\\
  &=\sqrt2\norm{G^{-1}}_{\cL(V_{2\delta,H})}
  e^{-t\lambda_k}\lambda_k^{\frac12+\delta+\gamma}\frac{1}{\sqrt{1-e^{-2\lambda_k t}}}\,,
\end{align*}
from which it follows that 
\[
  \norm{A^\gamma Q_t^{-1/2}S(t)}_{\cL(H,H)}\leq 
  \sqrt2\norm{G^{-1}}_{\cL(V_{2\delta,H})}\frac{1}{t^{\frac12+\delta+\gamma}}
  \left(\max_{r\geq0}\frac{r^{\frac12+\delta+\gamma}e^{-r}}{\sqrt{1-e^{-2r}}}\right)\,.
\]
Therefore, the desired conclusion follows.
\end{proof}

Thanks to the properties of the operator $Q_t$ stated in the previous lemma, we can deduce the following results on the semigroup $R$.
\begin{lem}
  \label{lem:OU}
  Assume \ref{H1}-\ref{H2}-\ref{H3}. Then, the semigroup $R$ is strong Feller. 
  More specifically, 
  there exists a constant $C_{R}>0$ such that,
  for every $t>0$ and $v\in\cB_b(H)$, it holds that 
  $R_tv\in C^\infty_b(H)$ and
  \begin{align}
    \label{SF1}
    \sup_{x\in H}|(R_tv)(x)| &\leq \sup_{x\in H}|v(x)|\,,\\
    \label{SF2}
    \sup_{x\in H}\|D(R_tv)(x)\|_H &\leq 
    \frac{C_{R}}{t^{\frac12+\delta}}\sup_{x\in H}|v(x)|\,.
  \end{align}
  Moreover, for every $v\in C^1_b(H)$, it holds that 
  \begin{align}
    \label{SF2_bis}
    \sup_{x\in H}\|D^2(R_tv)(x)\|_{\cL(H,H)} &\leq 
    \frac{C_{R}}{t^{\frac12+\delta}}\sup_{x\in H}\|Dv(x)\|_H\,.
  \end{align} 
  Furthermore, for every $\gamma\in (0, 1/2-\delta)$
  and $\eta\in(0, (1-2\delta)/(1+2\delta))$
   there exists constants
  $C_{R,\gamma},C_{R,\eta}>0$ such that, 
  for every $t>0$ and $v\in UC^0_b(H)$, it holds that 
  $D(R_tv)(H)\subset D(A^\gamma)$ and
  \begin{align}
  \label{SF3}
  \sup_{x\in H}\|A^\gamma D(R_tv)(x)\|_H &\leq 
    \frac{C_{R,\gamma}}{t^{\frac12+\delta+\gamma}}\sup_{x\in H}|v(x)|\,,\\
  \label{SF4}
  \norm{R_tv}_{C^{1,\eta}_b(H)} &\leq 
  \frac{C_{R,\eta}}{t^{(1+\eta)(\frac12+\delta)}}\sup_{x\in H}|v(x)|\,.
  \end{align}
\end{lem}

\begin{proof}[Proof of Lemma~\ref{lem:OU}]
 The proof of \eqref{SF1}--\eqref{SF2_bis}
 follows from   
 \eqref{strong1}--\eqref{strong2} and \cite[Thm.~4]{DPF}.
 Moreover, \eqref{SF3} follows from 
 \cite[Thm.~6.2.2]{dapratozab2} and \eqref{strong2}.
 Finally, \eqref{SF4} follows from \cite[Lem.~6.4.1]{dapratozab2}.
\end{proof}

The following preliminary result will be crucial in the sequel. Even if it is rather classical 
(see, e.g., \cite{ito}), we present here 
a self contained proof for the reader's convenience. 
\begin{lem}
  \label{lem:dense}
  Assume \ref{H1}-\ref{H2}-\ref{H3}. Let $E\subset H$ be Borel-measurable
  with $N_{Q_\infty}(E)=1$.
  Then $E$ is dense in $H$. 
\end{lem}
\begin{proof}
  First, note that $Q_\infty\in\cL^1(H,H)$ is a positive operator with 
  $\operatorname{ker}(Q_\infty)=\{0\}$ thanks to \ref{H3}.
  Hence, there exists an orthonormal system $(q_k)_{k\in\enne}$ of $H$
  and a sequence of eigenvalues $(\zeta_k)_k$ such that $Q_\infty q_k=\zeta_k q_k$
  for every $k\in\enne$. In particular, it holds (see \cite[Thm.~1.2.1]{dapratozab2}) that  
  $N_{Q_\infty}$ is the restriction to $H$ (identified with $\ell^2$) of the product measure 
  on $(\erre^\infty, \cB(\erre^\infty))$ given by
  \[
  \bigotimes_{k\in\enne}N_{\zeta_k}\,,
  \]
  where $N_{\zeta_k}$ denotes a Gaussian measure on $\erre$
  with variance $\zeta_k$,  for every $k\in\enne$. 
  Note that the nondegeneracy condition $\operatorname{ker}(Q_\infty)=\{0\}$
  ensures that $\zeta_k>0$ for every $k\in\enne$, so that each $N_{\zeta_k}$
  is absolutely continuous with respect to the Lebesgue measure.
  Now, by contradiction, we suppose that there exists $a\in H$ and $r>0$ such that 
  \[
  \{x\in H:\|x-a\|_H<r\}\subset H\setminus \overline E\,,
  \]
  where $\overline E$ denotes the closure of $E$. Then, since $N_{Q_\infty}(E)=1$, one has that 
  \[
  \|x-a\|_H\geq r\,, \quad\text{for $N_{Q_\infty}$-a.e.~}x\in H\,.
  \]
  Let now $\alpha>0$ be fixed. It follows that 
  \[
  \int_H e^{-\frac\alpha2\|x-a\|_H^2}\,N_{Q_\infty}(\d x) \leq e^{-\frac\alpha2r^2}\,.
  \]
  The left hand side can be explicitly computed by using the factorisation of $N_{Q_\infty}$
  and the fact that $\zeta_k>0$ for every $k\in\enne$. More precisely, 
  setting $x_k:=(x,q_k)_H$ and $a_k:=(a,q_k)$ for every $k\in\enne$ and $x\in H$,
  after classical computations (see \cite{ito} or \cite[Prop.~1.2.8]{dapratozab2}), we have
  \begin{align*}
  \int_H e^{-\frac\alpha2\|x-a\|_H^2}\,N_{Q_\infty}(\,d x)
  &=\int_H\prod_{k\in\enne}e^{-\frac\alpha2(x_k-a_k)^2}\,N_{Q_\infty}(\,d x)\\
  &=\prod_{k\in\enne}\int_\erre e^{-\frac\alpha2(x_k-a_k)^2}\,N_{\zeta_k}(\,d x_k)\\
  &=\prod_{k\in\enne}\int_\erre 
  \frac1{\sqrt{2\pi\zeta_k}}e^{-\frac\alpha2(x_k-a_k)^2 - \frac1{2\zeta_k}x_k^2}\,\d x_k\\
  &=\prod_{k\in\enne}(1+\alpha\zeta_k)^{-\frac12}\exp\left(-\frac\alpha2\frac{a_k^2}{1+\alpha\zeta_k}\right)\,.
  \end{align*}
  Putting everything together, we deduce that 
  \beq\label{ito_aux1}
  \exp\left[\frac12\sum_{k\in\enne}\left(\ln(1+\alpha\zeta_k)+\frac{\alpha a_k^2}{1+\alpha\zeta_k}\right)\right]
  \geq e^{\frac\alpha2r^2}\,, \quad\forall\,\alpha>0\,.
  \eeq
  Now, for every $N\in\enne$, one has 
  \begin{align*}
  \exp\left[\frac12\sum_{k\geq N}\left(\ln(1+\alpha\zeta_k)+\frac{\alpha a_k^2}{1+\alpha\zeta_k}\right)\right]
  &\leq\exp\left[\frac\alpha2
  \sum_{k\geq N}\left(\zeta_k+a_k^2\right)\right]\,.
  \end{align*}
  Since $(\zeta_k)_k\in\ell^1$ and $(a_k)_k\in\ell^2$, we can choose and fix $\bar N\in\enne$, 
  independent of $\alpha$, such that 
  \[
  \sum_{k\geq \bar N}\left(\zeta_k+a_k^2\right) \leq \frac{r^2}2\,,
  \]
  so that by using \eqref{ito_aux1} we get
  \[
  \exp\left[\frac12\sum_{k<\bar N}\left(\ln(1+\alpha\zeta_k)+\frac{\alpha a_k^2}{1+\alpha\zeta_k}\right)\right]
  \geq e^{\frac\alpha4r^2}\,, \quad\forall\,\alpha>0\,.
  \]
  It follows that 
  \begin{align*}
  1&\leq \left[e^{-\frac\alpha4r^2}\prod_{k<\bar N}(1+\alpha\zeta_k)^{\frac12}\right]
  \exp\left(\frac\alpha2\sum_{k<\bar N}\frac{a_k^2}{1+\alpha\zeta_k}\right)\\
  &\leq \left[e^{-\frac\alpha4r^2}\prod_{k<\bar N}(1+\alpha\zeta_k)^{\frac12}\right]
  \exp\left(\frac12\sum_{k<\bar N}\frac{a_k^2}{\zeta_k}\right)\,, \quad\forall\,\alpha>0\,,
  \end{align*}
  so that letting $\alpha\to+\infty$ yields a contradiction, and this concludes the proof.
\end{proof}

We are now ready to state the definition of mild solution for the 
Kolmogorov equation \eqref{eq_kolm}.
For convenience, we set $\tilde B:H\to D(A^{-\beta})$ as
\[
  \tilde B(x):=
  \begin{cases}
  B(x) \quad&\text{if } x\in V_{2\alpha}\,,\\
  0 \quad&\text{if } x\in H\setminus V_{2\alpha}\,.
  \end{cases}
\]

\begin{defin}
  \label{def:sol_kolm}
  A mild solution to the Kolmogorov equation \eqref{eq_kolm} is a 
  function $u\in C^1_b(H)$ such that $Du\in C^0_b(H;D(A^\beta))$ and
  \[
  u(x)=\int_0^{\infty}e^{-\lambda t} 
  R_t[f+\langle\tilde B,Du\rangle](x)\,\d t \quad\forall\,x\in H\,.
  \]
\end{defin}

\begin{prop}
  \label{prop:kolm}
  Assume \ref{H1}-\ref{H2}-\ref{H3}.
  Then, there exists $\lambda_0 > 0$ , only depending on $A, B, G$,
   such that, for every $\lambda>\lambda_0$,
  there exists a unique mild solution to
  the Kolmogorov equation \eqref{eq_kolm} in  the sense of Definition~\ref{def:sol_kolm}.
\end{prop}
\begin{proof}
  We use a fixed point argument. To this end, we introduce the space
  \[
  \mathcal U:=\left\{v\in C^1_b(H):\quad Dv\in C^0_b(H;V_{2\beta})\right\}
  \]
  with norm
  \[
  \norm{v}_{\mathcal U}:=\norm{v}_{C^1_b(H)} + \norm{Dv}_{C^0_b(H;V_{2\beta})}\,,
  \quad v\in\mathcal U\,,
  \]
  and define the operator 
  $\mathcal T_\lambda:\mathcal U\to C^0_b(H)$ as follows
  \[
  \mathcal T_\lambda v(x):=\int_0^{\infty}e^{-\lambda t} 
  R_t[f+\langle\tilde B,Dv\rangle](x)\,\d t\,,
  \quad v\in \mathcal U\,.
  \]
  Since $\beta\in[0,\frac12-\delta)$, 
  by Lemma~\ref{lem:OU} we have
  \begin{align*}
  e^{-\lambda t} \|DR_t[f+\langle\tilde B,Dv\rangle](x)\|_{2\beta}
  &=e^{-\lambda t} \|A^\beta DR_t[f+\langle\tilde B,Dv\rangle](x)\|_H\\
  &\leq 
  C_R\frac{e^{-\lambda t}}{t^{\frac12+\delta+\beta}}
  \sup_{y\in H}|f(y)+\langle\tilde B(y),Dv(y)\rangle|\\
  &\leq C_R\left(\|f\|_{C^0_b(H)} + 
  \sup_{y\in V_{2\alpha}}\norm{B(y)}_{-2\beta}
  \norm{v}_{\mathcal U}\right)
  \frac{e^{-\lambda t}}{t^{\frac12+\delta+\beta}}\,.
  \end{align*}
  Hence, by the dominated convergence theorem, we can differentiate 
  under integral sign, and we obtain that $\mathcal T_\lambda$ 
  is well-defined as an operator from $\mathcal U$ to $\mathcal U$ and 
  \[
  D(\mathcal T_\lambda v)(x)=
  \int_0^{\infty}e^{-\lambda t} DR_t[f+\langle\tilde B,Dv\rangle](x)\,\d t\,,
  \quad v\in \mathcal U\,.
  \]
  Moreover, by analogous computations, for every $v_1,v_2\in \mathcal U$, we have 
  \[
  |\mathcal T_\lambda v_1(x) - \mathcal T_\lambda v_2(x)|
  \leq \frac1\lambda\sup_{y\in V_{2\alpha}}\norm{B(y)}_{-2\beta} 
  \norm{v_1-v_2}_{\mathcal U} \,,
  \]
  and 
  \begin{align*}
  \|D\mathcal T_\lambda v_1(x) - D\mathcal T_\lambda v_2(x)\|_{2\beta}&\leq
  C_R\sup_{y\in V_{2\alpha}}\norm{B(y)}_{-2\beta}
  \norm{v_1-v_2}_{\mathcal U}
  \int_0^{+\infty}e^{-\lambda t} \frac1{t^{\frac12+\delta+\beta}}\,\d t\\
  &=\frac{C_R}{\lambda^{\frac12-\delta-\beta}}\Gamma\left(\frac12-\delta-\beta\right)
  \sup_{y\in V_{2\alpha}}\norm{B(y)}_{-2\beta}\norm{v_1-v_2}_{\mathcal U}\,.
  \end{align*}
  Hence, provided that $\lambda$ is chosen sufficiently large 
  (only depending on $B$, $\delta$, $\beta$, $R$) we infer that 
  $\mathcal T_\lambda: \mathcal U\to \mathcal U$ is a contraction. Hence, the 
  thesis follows by the classical Banach fixed point theorem.
\end{proof}

From now on, we fix $\lambda_0>0$ as in Proposition~\ref{prop:kolm}.

In order to prove uniqueness in law, we need a further regularisation 
of the solution $u$ to the Kolmogorov equation, which is at least of class $C^2_b(H)$
so that a suitable It\^o formula can be applied. 
The idea is to consider an approximated version of the Kolmogorov equation 
for which one is able to prove existence of strong solutions. To this end, we introduce a linear self-adjoint negative definite operator $C$
with $C^{-1}\in\cL^1(H,H)$ and denote by $(T(t))_{t\geq0}$ the associated linear semigroup on $H$. Then,
for every $\eps>0$, we define the regularised operator 
\begin{equation*}
  B_\eps:H\to H\,, \qquad
  B_\eps(x):=\int_H T(\eps)\tilde B(T(\eps)x+y)\,N_{\frac12C^{-1}(T(\eps)-I)}(\d y)\,,
  \quad x\in H\,,
\end{equation*}
and we also define 
\begin{equation*}
  f_\eps:H\to \erre\,, \qquad
  f_\eps(x):=\int_H f(T(\eps)x+y)\,N_{\frac12C^{-1}(T(\eps)-I)}(\d y)\,, 
  \quad x\in H\,,
\end{equation*}
where, as usual, $\smash{N_{\frac12C^{-1}(T(\eps)-I)}}$ denotes the Gaussian measure with mean $0$ and covariance operator $\smash{\frac12C^{-1}(T(\eps)-I)}$. Moreover, since $C^{-1}\in\cL^1(H,H)$, it follows (see \cite[Sec.~6.2]{dapratozab2}) that 
$B_\eps\in C^\infty_b(H;H)$ and $f_\eps\in C^\infty_b(H)$ for every $\eps>0$, and 
\beq
  \label{conv_eps_Bf}
  \lim_{\eps\to0^+}
  \|B_\eps(x)-\tilde B(x)\|_{-2\beta}
  =0
  \quad\forall\,x\in H\,,
  \qquad
  \lim_{\eps\to0^+}
  \|f_\eps-f\|_{C^0_b(H)}=0\,.
\eeq

Now, for every $\eps>0$, we consider the regularised Kolmogorov equation 
\beq
 \label{eq_kolm_reg}
 \lambda u_\eps(x) + Lu_\eps(x) = f_\eps(x) + (B_\eps(x),Du_\eps(x))\,, \quad x\in D(A)\,.
\eeq
\begin{defin}
  \label{def:sol_kolm_reg}
  Let $\eps>0$.
  A mild solution to the regularised Kolmogorov equation \eqref{eq_kolm_reg} is a 
  function $u_\eps\in C^1_b(H)$ such that 
  \[
  u_\eps(x)=\int_0^{\infty}e^{-\lambda t} 
  R_t[f_\eps+(B_\eps,Du_\eps)](x)\,\d t\,, \quad\forall\,x\in H\,.
  \]
  A strong solution to the regularised Kolmogorov equation \eqref{eq_kolm_reg} is a 
  function $u_\eps\in C^2_b(H)$ such that 
  $QD^2u_\eps(x)\in\cL^1(H,H)$ for all $x\in H$ and
  \[
  \lambda u_\eps(x)
  -\frac12\operatorname{Tr}(QD^2u_\eps(x)) + (Ax, Du_\eps(x))
  = f_\eps(x) + (B_\eps(x),Du_\eps(x))\,, \quad \forall\,x\in D(A)\,.
  \]
\end{defin}

\begin{prop}
  \label{prop:kolm_reg}
  Assume \ref{H1}-\ref{H2}-\ref{H3} and
  let $\eps>0$. Then, for every $\lambda>\lambda_0$
  there exists a unique strong solution
  $u_\eps$ to
  the regularised Kolmogorov equation \eqref{eq_kolm_reg}, 
  which is also a mild solution in  the sense of Definition~\ref{def:sol_kolm_reg},
  and it satisfies 
  \beq\label{cont_L1}
  Du_\eps\in C^0_b(H; D(A^\beta))\,,
  \qquad QD^2u_\eps \in C^0_b(H; \cL^1(H,H))\,.
  \eeq
\end{prop}
\begin{proof}
 First of all, we note that
 the same proof of Proposition~\ref{prop:kolm} provides 
 also a sequence $(u_\eps)_\eps\subset C^1_b(H)$ of mild solutions
 to the regularised Kolmogorov equation in the sense of 
 Definition~\ref{def:sol_kolm_reg}, with 
 $(Du_\eps)_\eps\subset C^0_b(H; D(A^\beta))$. In particular, one has that 
   \[
  u_\eps(x)=\int_0^{\infty}e^{-\lambda t} R_t[f_\eps
  +(B_\eps,Du_\eps)](x)\,\d t \,, \quad\forall\,x\in H\,.
  \]
  and 
  \[
  Du_\eps(x) = \int_0^{\infty}e^{-\lambda t} DR_t[f_\eps
  +(B_\eps,Du_\eps)](x)\,\d t \,, \quad\forall\,x\in H\,.
  \]
  Now, note that 
  for every $\eta\in(0, (1-2\delta)/(1+2\delta))$, thanks to \eqref{SF4} one has,
  for all $t>0$, that
  \begin{align*}
  &e^{-\lambda t}\norm{DR_t[f_\eps + (B_\eps,Du_\eps)]}_{C^{0,\eta}_b(H)}\\
  &\qquad\leq \frac{C_{R,\eta}}{t^{(1+\eta)(\frac12+\delta)}}e^{-\lambda t}
  \left(\norm{f}_{C^0_b(H)} + 
  \norm{B_\eps}_{C^0_b(H;H)}
  \norm{Du_\eps}_{C^0_b(H;H)}\right) \,,
  \end{align*}
  where $(1+\eta)(\frac12+\delta)\in (0,1)$. Hence, 
  the dominated convergence theorem
  yields $(u_\eps)_\eps\subset C^{1,\eta}_b(H)$.
 In order to show that $(u_\eps)_\eps$ are actually strong solutions 
in the sense of Definition~\ref{def:sol_kolm_reg}, we
use hypothesis \ref{H3} and
apply the result in \cite[Prop.~6.4.2]{dapratozab2}. To this end, 
we only need to check that 
\[
  x\mapsto f_\eps(x) + (B_\eps(x), Du_\eps(x))
\]
is H\"older continuous and bounded in $H$. 
Clearly, this is true for $f_\eps$.
As far as the second term is concerned, 
since $(u_\eps)_\eps\subset C^{1,\eta}(H)$ for all $\eta\in(0, (1-2\delta)/(1+2\delta))$,
recalling that $B_\eps\in C^1_b(H;H)$, 
for every $x,y\in H$ we have
\begin{align*}
  &\left|(B_\eps(x), Du_\eps(x))- (B_\eps(y), Du_\eps(y))\right|\\
  &\qquad\leq\left|(B_\eps(x), Du_\eps(x)- Du_\eps(y))\right|+
  \left|(B_\eps(x)-B_\eps(y), Du_\eps(y))\right|\\
  &\qquad\leq\norm{B_\eps}_{C^0_b(H;H)}
  \norm{u_\eps}_{C^{1,\eta}(H)}
  \norm{x-y}^\eta + 
  \norm{B_\eps}_{C^{0,1}(H;H)}
  \norm{u_\eps}_{C^{1}_b(H)}\norm{x-y}\\
  &\qquad\leq C_\eps\left(\norm{x-y}^\eta + \norm{x-y}\right)\,.
\end{align*}
This shows indeed that $x\mapsto f_\eps(x) + (B_\eps(x), Du_\eps(x))$
is H\"older continuous for every $\eps>0$.
Consequently, 
by \cite[Prop.~6.4.2]{dapratozab2} 
we deduce that, for every $\eps>0$, it holds also that 
\[
u_\eps\in C^{\eta+\frac{2}{1+2\delta}}_b(H)\,, \quad\forall\,\eta\in (0, (1-2\delta)/(1+2\delta))\,.
\]
Since $\frac{1-2\delta}{1+2\delta}+\frac{2}{1+2\delta}=1+2\frac{1-2\delta}{1+2\delta}$, 
this implies that 
\[
  u_\eps\in C^{1,\sigma}_b(H)\,, \quad\forall\,\sigma\in (0, 2(1-2\delta)/(1+2\delta))\,.
\]
By a standard bootstrap argument, analogous computations show that 
\[
x\mapsto f_\eps(x) + (B_\eps(x), Du_\eps(x)) \in C^{0,\sigma}(H)\,,
\quad\forall\,\sigma\in (0, 2(1-2\delta)/(1+2\delta))\,,
\]
and applying again \cite[Prop.~6.4.2]{dapratozab2}
we can refine the regularity of $u_\eps$ as
\[
u_\eps\in C^{\sigma+\frac{2}{1+2\delta}}_b(H)\,, 
\quad\forall\,\sigma\in (0, 2(1-2\delta)/(1+2\delta))\,.
\]
By noticing that $2\frac{1-2\delta}{1+2\delta}+\frac{2}{1+2\delta}
=\frac{4-4\delta}{1+2\delta}=
1+3\frac{1-2\delta}{1+2\delta}$, an iteration argument shows that 
\[
u_\eps\in C^{\ell_k}_b(H)\,, \quad\forall\,k\in\enne\,,
\]
where
\[
  \ell_k=1+\min\left\{k\frac{1-2\delta}{1+2\delta}, \frac2{1+2\delta}\right\}\,, \quad k\in\enne\,.
\]
Hence, for $k=k(\delta)$ sufficiently large one obtains the optimal regularity
\[
u_\eps\in C^{1+\frac2{1+2\delta}}(H)\,.
\]
Noting that $1+\frac2{1+2\delta}>2$
for all $\delta\in(0,\frac12)$,
this implies in particular that $u_\eps\in C^2_b(H)$.
Moreover, hypothesis \eqref{eq:L4} and
the same result \cite[Prop.~6.4.2]{dapratozab2} ensure also that
$QD^2u_\eps(x)\in\cL^1(H,H)$ for all $x\in H$, 
so that $u_\eps$ is a strong solution to the regularised 
Kolmogorov equation in the sense of 
Definition~\ref{def:sol_kolm_reg}.\\
Eventually, we show that, for every $\eps>0$, it holds that 
$QD^2u_\eps \in C^0_b(H;\cL^1(H,H))$. To this end, 
for every $\varphi\in\cB_b(H)$ and $t>0$, 
we introduce the linear bounded operator $\Psi_\varphi(t)\in\cL(H,H)$ as
\[
  (\Psi_\varphi(t) h, k):=\int_H
  \left[(h, Q_t^{-\frac12}y)(k, Q_t^{-\frac12}y)-(h,k)\right]
  \varphi(y)\,N_{Q_t}(\d y)\,, \quad h,k\in H\,.
\]
By \cite[Lem.~6.2.7]{dapratozab2} it holds that 
\[
\Psi_\varphi(t)\in\cL^2(H,H)\,, \qquad
\norm{\Psi_\varphi(t)}_{\cL^2(H,H)}\leq 2\sup_{x\in H}|\varphi(x)|
\quad\forall\,\varphi\in\cB_b(H)\,, \quad\forall\,t>0\,.
\]
By linearity, this shows in particular 
that for every $t>0$, 
$\varphi\mapsto \Psi_\varphi(t)\in \cL(C^0_b(H); \cL^2(H,H))$. 
Now, let us set for convenience $\varphi_\eps:=f_\eps+(B_\eps,Du_\eps)\in C^0_b(H)$. By Proposition~\cite[Prop.~6.2.8]{dapratozab2} one has that,
for every $t>0$, 
\[
  GD^2(R_t\varphi_\eps)(x)G
  =
  Q_t^{-\frac12}S(t)G\Psi_{\varphi_\eps(S(t)x+\cdot)}(t)
  Q_t^{-\frac12}S(t)G \quad\forall\,x\in H\,.
\]
We note that for every $(x_n)_n,x\subset H$ such that $x_n\to x$ in $H$, 
since $\varphi_\eps\in C^1_b(H)$ it holds that 
$\varphi_\eps(S(t)x_n+\cdot)\to \varphi_\eps(S(t)x+\cdot)$ in $C^0_b(H)$.
Hence, recalling that $\varphi\mapsto \Psi_\varphi(t)\in \cL(C^0_b(H); \cL^2(H,H))$
one has that 
\[
  x\mapsto \Psi_{\varphi_\eps(S(t)x+\cdot)}(t)
  \in C^0(H; \cL^2(H,H))\,,
\]
which in turn implies, thanks to the fact that 
that $Q_t^{-\frac12}S(t)G\in \cL^4(H,H)$ 
for almost every $t>0$ by \ref{H3}, that 
\[
  x\mapsto Q_t^{-\frac12}S(t)G\Psi_{\varphi_\eps(S(t)x+\cdot)}(t)
  Q_t^{-\frac12}S(t)G \in C^0(H; \cL^1(H,H))
  \quad\text{for a.e.~}t>0\,.
\]
This implies that
\[
  x\mapsto GD^2(R_t\varphi_\eps)(x)G \in C^0(H; \cL^1(H,H))
  \quad\text{for a.e.~}t>0\,.
\]
Moreover, we note that 
\[
  \norm{GD^2(R_t\varphi_\eps)(x)G}_{\cL^1(H,H)}
  \leq 2\norm{Q_t^{-\frac12}S(t)G}^2_{\cL^4(H,H)}\norm{\varphi_\eps}_{C^0_b(H)}
  \quad\text{for a.e.~}t>0\,.
\]
Moreover, by \cite[Prop.~6.2.9]{dapratozab2} and the fact that 
$\varphi_\eps\in UC_b^1(H)$ one has also that 
\[
  \norm{GD^2(R_t\varphi_\eps)(x)G}_{\cL^1(H,H)}
  \leq \norm{Q_t^{-\frac12}S(2t)Q}_{\cL^2(H,H)}\norm{\varphi_\eps}_{C^1_b(H)}
  \quad\text{for a.e.~}t>0\,.
\]
By interpolation (see \cite[Prop.~6.4.2]{dapratozab2}) we infer that 
\[
  \norm{GD^2(R_t\varphi_\eps)(x)G}_{\cL^1(H,H)}
  \leq 2^{1-\vartheta}\norm{Q_t^{-\frac12}S(t)G}^{2(1-\vartheta)}_{\cL^4(H,H)}
  \norm{Q_t^{-\frac12}S(2t)Q}_{\cL^2(H,H)}^\vartheta
  \norm{\varphi_\eps}_{C^{0,\vartheta}_b(H)}
\]
for almost every $t>0$. Thanks to assumption \ref{H3},
the dominated convergence theorem, and the fact that $u_\eps$
is a mild solution to the regularised Kolmogorov equation,  this implies that 
\[
  x\mapsto QD^2u_\eps(x) \in C^0_b(H; \cL^1(H,H))\,,
\]
and this concludes the proof.
\end{proof}

\begin{prop}
  \label{prop:kolm_reg_est}
  Assume \ref{H1}-\ref{H2}-\ref{H3}. 
  Let $\lambda>\lambda_0$,  
  and let $(u_\eps)_\eps\subset C^2_b(H)$
  be the strong solutions to the Kolmogorov equation
  \eqref{eq_kolm_reg}.
  Then, there exists a constant $C>0$, 
  only depending on $\lambda, \beta, A, B, G, f$, such that
  \begin{align}
  \label{est1_eps}
  \norm{u_\eps}_{C^1_b(H)} + \norm{Du_\eps}_{C^0_b(H; D(A^\beta))}&\leq C\,.
  \end{align}
  Moreover, for every 
  $\eta\in(0,(1-2\delta)/(1+2\delta))$,
  there exists a constant
  $C_{\eta}>0$, only depending on $\lambda, \beta, A, B, G, f$, such that
  \begin{align}
  \label{est3_eps}
  \norm{u_\eps}_{C^{1,\eta}_b(H)}&\leq C_\eta\,. 
  \end{align}
\end{prop}
\begin{proof}
  We first note that 
  \[
  u_\eps(x)=\int_0^{\infty}e^{-\lambda t} R_t[f_\eps
  +(B_\eps,Du_\eps)](x)\,\d t \,, \quad\forall\,x\in H\,.
  \]
  Moreover, arguing as in the proof of Proposition~\ref{prop:kolm}, one has that 
  \beq
  \label{Du_eps}
  Du_\eps(x) = \int_0^{\infty}e^{-\lambda t} DR_t[f_\eps
  +(B_\eps,Du_\eps)](x)\,\d t \,, \quad\forall\,x\in H\,.
  \eeq
  Now, by \eqref{SF2}, \eqref{SF3},
  the properties of convolution, the fact that $T$ is a contraction on $V_{-2\beta}$,
  and the dominated convergence theorem,
  it follows
  that 
  \begin{align*}
  \norm{A^\beta Du_\eps}_{C^0_b(H;H)}&\leq
  \int_0^\infty e^{-\lambda t}\frac{C_R}{t^{\frac12+\delta+\beta}}
  \sup_{x\in H}\left(|f_\eps(x)| + \norm{B_\eps(x)}_{-2\beta}
  \norm{Du_\eps(x)}_{2\beta}\right)\,\d t\\
  &\leq C_R\left(\norm{f}_{C^0_b(H)} + 
  \sup_{y\in V_{2\alpha}}\norm{B(y)}_{-2\beta}
  \norm{Du_\eps}_{C^0_b(H;V_{2\beta})}\right)
  \int_0^\infty e^{-\lambda t}\frac{1}{t^{\frac12+\delta+\beta}}\,\d t\\
  &\leq \frac{C_R}{\lambda^{\frac12-\delta-\beta}}
  \Gamma\left(\frac12-\delta-\beta\right)
  \left(\norm{f}_{C^0_b(H)} + 
  \sup_{y\in V_{2\alpha}}\norm{B(y)}_{-2\beta}
  \norm{Du_\eps}_{C^0_b(H;V_{2\beta})}\right)\,.
  \end{align*}
  Now, noting that the choice $\lambda>\lambda_0$ yields 
  \[
  \frac{C_R}{\lambda^{\frac12-\delta-\beta}}
  \Gamma\left(\frac12-\delta-\beta\right)
  \sup_{y\in V_{2\alpha}}\norm{B(y)}_{-2\beta}<1\,,
  \]
  the first estimate \eqref{est1_eps} follows. 
  Eventually, 
  for every $\eta\in(0, (1-2\delta)/(1+2\delta))$, by virtue of \eqref{SF4} one has,
  for all $t>0$,
  \begin{align*}
  &e^{-\lambda t}\norm{DR_t[f_\eps + (B_\eps,Du_\eps)]}_{C^{0,\eta}_b(H)}\\
  &\qquad\leq \frac{C_{R,\eta}}{t^{(1+\eta)(\frac12+\delta)}}e^{-\lambda t}
  \left(\norm{f}_{C^0_b(H)} + 
  \sup_{y\in V_{2\alpha}}\norm{B(y)}_{-2\beta}
  \norm{Du_\eps}_{C^0_b(H;V_{2\beta})}\right) \,,
  \end{align*}
  where $(1+\eta)(\frac12+\delta)\in (0,1)$. Hence, 
  the dominated convergence theorem and the estimate \eqref{est1_eps}
  yield that $u_\eps\in C^{1,\eta}_b(H)$ and 
  \begin{align*}
    \norm{u_\eps}_{C^{1,\eta}_b(H)}\leq 
    C_{R,\eta}
    \left(\norm{f}_{C^0_b(H)} + 
    C\sup_{y\in V_{2\alpha}}\norm{B(y)}_{-2\beta}\right)
    \int_0^\infty\frac{1}{t^{(1+\eta)(\frac12+\delta)}}e^{-\lambda t}\,\d t\,,
  \end{align*}
  which proves the estimate \eqref{est3_eps}.
\end{proof}

\begin{prop}
  \label{prop:kolm_reg_conv}
  Assume \ref{H1}-\ref{H2}-\ref{H3}. Let $\lambda>\lambda_0$, 
  let $(u_\eps)_\eps\subset C^2_b(H)$
  be the strong solutions to the Kolmogorov equation
  \eqref{eq_kolm_reg}, and let $u\in C^1_b(H)$ be the mild solution 
  to the Kolmogorov equation \eqref{eq_kolm}.
  Then,
  for all 
  $\eta\in (0,(1-2\delta)/(1+2\delta))$ and $\gamma\in[0, \beta)$, it holds that
  $u\in C^{1,\eta}_b(H)$
  and
  \begin{align}
  \label{conv1_eps}
  \lim_{\eps\to0^+}
  |u_\eps(x)-u(x)| = 0\,, \quad &\forall\,x\in H\,,\\
  \label{conv2_eps}
  \lim_{\eps\to0^+}
  \norm{Du_\eps(x) - Du(x)}_{2\gamma}=0\,, \quad &\forall\,x\in H\,.
  \end{align}
\end{prop}
\begin{proof}
The main idea of the proof is to first extract a candidate limit 
$\tilde u$ for $(u_\eps)_\eps$
by using compactness methods. Secondly, exploiting the continuity of 
the semigroup $R$ in $L^2(H,N_{Q_{\infty}})$ we show that 
$\tilde u$ coincides with $u$, $N_{Q_{\infty}}$-almost surely.
Finally, by the density result given in Lemma~\ref{lem:dense},
we conclude that actually $u=\tilde u$ on $H$.

For every $\ell\in\enne_+$ let $(\mathcal B_\ell)_{\ell>0}$ be a family 
of compact subsets of $H$ such that $\mathcal B_{\ell_1}\subset \mathcal B_{\ell_2}$
for all $\ell_1<\ell_2$ and with $\cup_{\ell>0}\mathcal B_\ell$ dense in $H$:
for example, we can take
\[
  \mathcal B_\ell:=\left\{x \in V_{2\alpha}: \norm{x}_{2\alpha}\leq \ell\right\}
\]
and note that $\mathcal B_\ell$ is a compact subset of $H$. Now, 
let us fix $\ell\in\enne_+$ and set, for all $\eps>0$,
$u_\eps^\ell:=(u_\eps)_{|\mathcal B_\ell}: \mathcal B_\ell\to \erre$.
By the estimate \eqref{est1_eps}
the sequence $(u_\eps^\ell)_\eps$
is uniformly equicontinuous and bounded. Hence, 
the Arzel\`a--Ascoli theorem yields the existence of
a subsequence 
$(u_{\eps_n}^\ell)_n$ and $u^\ell\in C^0_b(\mathcal B_\ell)$ such that 
\[
  \lim_{n\to\infty}\sup_{x\in \mathcal B_\ell}|u_{\eps_n}^\ell(x) - u^\ell(x)|=0\,.
\]
It is straightforward to check that the family $(u^\ell)_\ell$ is consistent
on the family of balls $(\mathcal B_\ell)_\ell$ in the sense that 
$u^{\ell+1}=u^\ell$ on $\mathcal B_\ell$ for all $\ell\in\enne_+$.
Consequently, this procedure identifies an element $\tilde u: V_{2\alpha}\to\erre$
such that, by possibly employing a diagonal argument, 
\beq
  \label{conv1_aux}
  \lim_{n\to\infty}\sup_{x\in \mathcal B_\ell}|u_{\eps_n}(x) - \tilde u(x)|=0
  \quad\forall\,\ell\in\enne_+\,,
\eeq
hence, in particular, 
\[
  \lim_{n\to\infty}|u_{\eps_n}(x) - \tilde u(x)|=0
  \quad\forall\,x\in V_{2\alpha}\,.
\]
We show now that actually $\tilde u$ can be uniquely extended to 
a Lipschitz function on the whole $H$. To this end, by \eqref{est1_eps},
for every $x_1,x_2\in V_{2\alpha}$ and $n\in\enne$ we have 
\begin{align*}
  |\tilde u(x_1)-\tilde u(x_2)|
  &\leq 
  |\tilde u(x_1) - u_{\eps_n}(x_1)|
  +|u_{\eps_n}(x_1) - u_{\eps_n}(x_2)|
  +|u_{\eps_n}(x_2)-\tilde u(x_2)|\\
  &\leq C\norm{x_1-x_2}
  +|\tilde u(x_1) - u_{\eps_n}(x_1)|
  +|u_{\eps_n}(x_2)-\tilde u(x_2)|\,,
\end{align*}
so that letting $n\to\infty$ yields 
\[
  |\tilde u(x_1)-\tilde u(x_2)|\leq C\norm{x_1-x_2}\,, \quad\forall\,x_1,x_2\in V_{2\alpha}\,.
\]
Since $V_{2\alpha}$ is dense in $H$, this shows that actually $\tilde u$
uniquely extends to an element $\tilde u \in C^{0,1}(H)$.
Moreover, by a similar argument one can show the pointwise
convergence of $(u_{\eps_n})_n$ to $\tilde u$ on the whole $H$. Indeed, 
for every $x\in H$, there is $(x_j)_j\subset V_{2\alpha}$ such that 
$x_j\to x$ in $H$ as $j\to\infty$. Hence, by Lipschitz-continuity and \eqref{est1_eps}, we get
\begin{align*}
 |\tilde u(x) - u_{\eps_n}(x)|
 &\leq|\tilde u(x) - \tilde u(x_j)|
 +|\tilde u(x_j) - u_{\eps_n}(x_j)|
 +|u_{\eps_n}(x_j)-u_{\eps_n}(x)|\\
 &\leq 2C\norm{x_j-x} + |\tilde u(x_j) - u_{\eps_n}(x_j)|\,,
\end{align*}
which implies that 
\[
  \lim_{n\to\infty}|u_{\eps_n}(x)-\tilde u(x)|=0\,, \quad\forall\,x\in H\,.
\]
Let us consider now the sequence of derivatives $(Du_{\eps})_\eps$.
For all $\ell\in\enne_+$
we set $\mathfrak u_\eps^{\ell}:=(Du_\eps)_{|\mathcal B_\ell}:\mathcal B_\ell\to H$.
Proceeding as before, given a fixed $\ell\in\enne_+$,
one has that 
$(\mathfrak u^\ell_{\eps})_\eps$ is uniformly equicontinuous thanks to \eqref{est3_eps},
and for all $x\in \mathcal B_\ell$ the sequence $(\mathfrak u_\eps^\ell(x))_\eps$
is relatively compact in $V_{2\gamma}$ thanks to \eqref{est3_eps}.
Hence, using again the Arzel\`a--Ascoli theorem and proceeding as above
possibly employing a diagonal argument and suitably renominating the subsequence, 
this identifies an element $\mathfrak u: V_{2\alpha}\to H$ such that 
\beq
\label{conv2_aux}
  \lim_{n\to\infty}\sup_{x\in \mathcal B_\ell}
  \norm{Du_{\eps_n}(x) - \mathfrak u(x)}_{2\gamma}=0\,,
  \quad\forall\,\ell\in\enne_+\,,
\eeq
hence also
\[
  \lim_{n\to\infty}\norm{Du_{\eps_n}(x) - \mathfrak u(x)}_{2\gamma}
  =0\,, \quad\forall\,x\in V_{2\alpha}\,.
\]
Arguing on the line of the previous computations for $\tilde u$, we can show that 
$\mathfrak u$ uniquely extends to a function $\mathfrak u\in C^{0,\eta}(H;H)$,
where $\eta$ is arbitrary and fixed in $(0, (1-2\delta)/(1+2\delta))$.
Indeed, for all $x_1,x_2\in V_{2\alpha}$ 
and $n\in\enne$, we have 
\begin{align*}
  \norm{\mathfrak u(x_1)-\mathfrak u(x_2)}
  &\leq 
  \norm{\mathfrak u(x_1) - Du_{\eps_n}(x_1)}
  +\norm{Du_{\eps_n}(x_1) - Du_{\eps_n}(x_2)}
  +\norm{Du_{\eps_n}(x_2)-\mathfrak u(x_2)}\\
  &\leq C\norm{x_1-x_2}^\eta
  +\norm{\mathfrak u(x_1) - Du_{\eps_n}(x_1)}
  +\norm{Du_{\eps_n}(x_2)-\mathfrak u(x_2)}\,,
\end{align*}
so that letting $n\to\infty$ yields 
\[
  \norm{\mathfrak u(x_1)-\mathfrak u(x_2)}
  \leq C\norm{x_1-x_2}^\eta \quad\forall\,x_1,x_2\in V_{2\alpha}\,,
\]
and the desired extension follows by density of $V_{2\alpha}$ in $H$.
Moreover, arguing as for $\tilde u$, it is immediate to show the pointwise convergence 
\[
  \lim_{n\to\infty}\norm{Du_{\eps_n}(x) - \mathfrak u(x)}_{2\gamma}
  =0\,, \quad\forall\,x\in H\,.
\]
To conclude, it remains to prove that $\tilde u \in C^{1,\eta}_b(H)$ and 
$D\tilde u=\mathfrak u$. To this end, from \eqref{conv1_aux} and \eqref{conv2_aux}
it follows that, for every $\ell\in\enne_+$, 
the sequence $((u_{\eps_n})_{|\mathcal B_\ell})_n$ is Cauchy in
the space 
\[ 
\left\{v\in C^1_b(\mathcal B_\ell): Dv\in C^0_b(\mathcal B_\ell; H)\right\}.
\] 
Consequently, 
one has that $\tilde u \in C^1_b(\mathcal B_\ell)$ for every $\ell\in\enne_+$
and $D\tilde u(x) = \mathfrak u(x)$ in $H$ for every $x\in V_{2\alpha}$.
Let now $x,h\in H$ and 
$(x_k)_k, (h_k)_k\subset V_{2\alpha}$ be such that $x_k\to x$ 
and $h_k\to h$ in $H$
as $k\to\infty$. Then, we have 
\begin{align*}
  \tilde u(x_k+ h_k)-\tilde u(x_k)
  =\int_0^1 \left(\mathfrak u(x_k+\sigma h_k), h_k\right)_{H}\,\d \sigma \,,
  \quad\forall\,k\in\enne\,,
\end{align*}
so that letting $k\to\infty$ yields, by the dominated convergence theorem, 
\[
  \tilde u(x+ h)-\tilde u(x) = \int_0^1\left(\mathfrak u(x+\sigma h), h\right)_{H}\,\d \sigma\,,
  \qquad\forall\,x,h\in H\,,
\]
showing that $\tilde u:H\to\erre$ 
is G\^ateaux differentiable with $D\tilde u=\mathfrak u\in C^{0,\eta}(H;H)$.
This implies then that $\tilde u\in C^{1,\eta}(H)$, as required. 
The boundedness of $\tilde u$ and
$D\tilde u$ follows from \eqref{est1_eps}--\eqref{est3_eps}
and the pointwise convergences obtained above.

Finally, note that 
\[
  u_{\eps_n}(x)=\int_0^{\infty}e^{-\lambda t} 
  R_t[f_{\eps_n}+(B_{\eps_n}, Du_{\eps_n})](x)\,\d t \,, \quad\forall\,x\in H\,,
\]
and that 
$f_{\eps_n}+(B_{\eps_n}, Du_{\eps_n})\to f+\langle\tilde B, Du\rangle$ 
in $L^2(H,N_{Q_\infty})$
by the dominated convergence theorem.
Hence, recalling that $R$ is a strongly continuous semigroup of contractions 
on $L^2(H,N_{Q_\infty})$, by letting $n\to\infty$ one obtains
\[
  \tilde u=\int_0^{\infty}e^{-\lambda t} 
  R_t[f+\langle\tilde B, D\tilde u\rangle]\,\d t \quad\text{in } L^2(H,N_{Q_\infty})\,.
\]
Now, by uniqueness of the mild solution 
$u$ of the Kolmogorov equation \eqref{eq_kolm}
in $L^2(H,N_{Q_\infty})$, 
this implies that $\tilde u = u$ $N_{Q_\infty}$-almost surely.
By Lemma~\ref{lem:dense} and the fact that $\tilde u\in C^{0,1}(H)$,
this implies that $\tilde u(x) = u(x)$ for all $x\in H$, hence in particular that 
$\tilde u$ is bounded and the convergences hold along the entire sequence $\eps$.
\end{proof}


\section{Proofs of main theorems}
\label{sec:uniq}
\subsection{Proof of Theorem~\ref{thm2}}\label{ssec:uniq_bound}
Let $x\in H$, and let 
\[
 \left(\Omega, \cF, (\cF_t)_{t\geq0}, \P, W, X\right)\,, \qquad
 \left(\Omega, \cF, (\cF_t)_{t\geq0}, \P, W, Y\right)
\]
be two global weak solutions to \eqref{eq0} in the sense of Definition~\ref{def_sol}
with respect to the same initial datum $x$. Moreover, 
let $\lambda>\lambda_0$ (as specified in Proposition~\ref{prop:kolm}), 
let $(u_\eps)_\eps\subset C^2_b(H)$
be the strong solutions to the Kolmogorov equation
\eqref{eq_kolm_reg}, and let $u\in C^1_b(H)$ be the mild solution 
to the Kolmogorov equation \eqref{eq_kolm}.

Note that $(e_k)_{k\in\enne}$ is a complete orthonormal system of $H$
made of eigenvectors of $A$, and $(\lambda_k)_{k\in\enne}$ are the corresponding 
eigenvalues. By setting $H_n:=\operatorname{span}\{e_1,\ldots,e_n\}$
and $P_n:H\to H_n$ 
as the orthogonal projection on $H_n$, for all $n\in\enne$,
we note that $P_n$ can be extended to a linear continuous operator 
$P_n:D(A^{-\beta})\to H_n$ as follows
\[
  P_nv:=\sum_{k=1}^n\langle v, e_k\rangle e_k\,, \quad v\in D(A^{-\beta})\,.
\]

Now, for every $j\in\enne$, we set $X_j:=P_jX$, $x_j:=P_jx$, 
$B_j:=P_jB$, $G_j:=P_jG$, and $Q_j:=G_jG_j^*$. Then,
since $H_j\subset D(A)$ for every $j\in\enne$, one has that $X_j$ are 
analytically strong solutions to the equation 
\begin{align*}
  X_j(t)+\int_0^tAX_j(s)\,\d s = x_j + \int_0^tB_j(X(s))\,\d s + \int_0^tG_j(s)\,\d W(s)
  \quad\text{in } H\,, \quad\forall\,t\geq0\,, \quad\P\text{-a.s.}
\end{align*}
By It\^o's formula and by adding and subtracting suitable terms one has that
\begin{align*}
  &\E\left[u_\eps(X_j(t))\right] 
  -\frac12\E\int_0^t\operatorname{Tr}\left[QD^2u_\eps(X_j(s))\right]\,\d s
  +\E\int_0^t\left(AX_j(s), Du_\eps(X_j(s))\right)\,\d s\\
  &=u_\eps(x_j)+ \E\int_0^t\left(B_j(X(s)), Du_\eps(X_j(s))\right)\,\d s
  +\frac12\E\int_0^t\operatorname{Tr}\left[(Q_j-Q)D^2u_\eps(X_j(s))\right]\,\d s\,,
\end{align*}
so that by the Kolmogorov equation \eqref{eq_kolm_reg} we get
\begin{align}
  \label{ito_uniq_aux}
  &\E\left[u_\eps(X_j(t))\right] - u_\eps(x_j) -\lambda\E\int_0^tu_\eps(X_j(s))\,\d s
  +\E\int_0^tf_\eps(X_j(s))\,\d s\\
  \nonumber
  &= \underbrace{\E\int_0^t\left(B_j(X(s))-B_\eps(X_j(s)), 
  Du_\eps(X_j(s))\right)\,\d s}_{I_1(\eps,j)}
  +\underbrace{\frac12\E\int_0^t\operatorname{Tr}\left[(Q_j-Q)D^2
  u_\eps(X_j(s))\right]\,\d s}_{I_2(\eps,j)}\,.
\end{align}
We now show that, by letting $\eps\to0^+$ and $j\to\infty$ under a suitable scaling, one obtains that $I_1(\eps,j)\to0$ and $I_2(\eps,j)\to0$.
To this end, by \eqref{est1_eps} we have
\begin{align*}
  &|I_1(\eps,j)|\\
  & \leq\E\int_0^t\left|\langle B_j(X(s))-B(X(s)), 
  Du_\eps(X_j(s))\rangle\right|\,\d s +
  \E\int_0^t\left|\langle B(X(s))-B_\eps(X(s)), 
  Du_\eps(X_j(s))\rangle\right|\,\d s\\
  &\qquad \qquad+\E\int_0^t\left|\langle B_\eps(X(s))-B_\eps(X_j(s)), 
  Du_\eps(X_j(s))\rangle\right|\,\d s\\
  &\leq C\left(\norm{B_j(X) - B(X)}_{L^2(\Omega; L^2(0,t; V_{-2\beta}))}
  +\norm{B_\eps(X) - B(X)}_{L^2(\Omega; L^2(0,t; V_{-2\beta}))}\right.\\
  &\qquad \qquad\left.
  +\norm{B_\eps(X) - B_\eps(X_j)}_{L^2(\Omega; L^2(0,t; V_{-2\beta}))}\right) \,.
\end{align*}

As far as $I_2$ is concerned, note that since $Q$ and $A$ commute, 
\begin{align*}
  I_2(\eps,j)&=
  \frac12\E\int_0^t\operatorname{Tr}\left[(Q_j-Q)D^2
  u_\eps(X_j(s))\right]\,\d s
  =\frac12\E\int_0^t\operatorname{Tr}\left[(P_j-I)
  QD^2u_\eps(X_j(s))\right]\,\d s\\
  &=\frac12\E\int_0^t\operatorname{Tr}\left[(P_j-I)
  GD^2u_\eps(X_j(s))G\right]\,\d s
\end{align*}
Hence, since $\norm{P_j-I}_{\cL(H,H)}\leq1$
one has
\begin{align*}
  &\E\int_0^t\operatorname{Tr}\left[(P_j-I)
  GD^2u_\eps(X_j(s))G\right]\,\d s\\
  &=\E\int_0^t\operatorname{Tr}\left[(P_j-I)
  GD^2u_\eps(X(s))G\right]\,\d s\\
  &\qquad+
  \E\int_0^t\operatorname{Tr}\left[(P_j-I)
  \left(GD^2u_\eps(X_j(s))G- GD^2u_\eps(X(s))G\right)\right]\,\d s\\
  &\leq
  \E\int_0^t\operatorname{Tr}\left[(P_j-I)
  GD^2u_\eps(X(s))G\right]\,\d s\\
  &\qquad+\E\int_0^t
  \norm{GD^2u_\eps(X_j(s))G- GD^2u_\eps(X(s))G}_{\cL^1(H,H)}
  \,\d s
\end{align*}
Since $P_j\to I$ in the strong operator topology of $\cL(H,H)$, 
i.e., $P_jx\to x$ for every $x\in H$ as $j\to\infty$, 
and since $QD^2u_\eps(x)\in \cL^1(H,H)$ for all $x\in H$, 
we infer that 
\[
  \operatorname{Tr}\left[(P_j-I)GD^2u_\eps(x)G\right]
  \to 0 \quad\text{in } \cL^1(H,H) \quad\forall\,x\in H\,.
\] 
This implies by the dominated convergence theorem that 
\[
  \lim_{j\to\infty}\E\int_0^t\operatorname{Tr}\left[(P_j-I)
  GD^2u_\eps(X(s))G\right]\,\d s =0
  \quad\forall\,\eps>0\,.
\]
Moreover, by exploiting \eqref{cont_L1} and again the dominated convergence theorem
we get 
\[
  \lim_{j\to\infty}\E\int_0^t
  \norm{GD^2u_\eps(X_j(s))G- GD^2u_\eps(X(s))G}_{\cL^1(H,H)}
  \,\d s =0
  \quad\forall\,\eps>0\,.
\]
Putting everything together we infer that 
\[
  \lim_{j\to\infty} I_2(\eps,j)=0 \quad\forall\,\eps>0\,.
\]
Note that we are only using that 
$u_\eps\in C^2_b(H)$ for every $\eps>0$, 
not that $(u_\eps)_\eps$ is bounded in $C^2_b(H)$, which is not the case.

Now, by \eqref{conv_eps_Bf}, the dominated convergence theorem,
and the fact that $\tilde B(X)=B(X)$ in $V_{-2\beta}$ almost everywhere,
for every $k\in\enne$, we can find $\eps_k>0$ such that 
\[
 \norm{B_{\eps_k}(X) - B(X)}_{L^2(\Omega; L^2(0,t; V_{-2\beta}))}\leq\frac1k\,.
\]
Given such $\eps_k$ and noting that, by definition of $P_j$,
the continuity of $B_{\eps_k}$, and the computations above
\begin{align*}
  &\lim_{j\to\infty}
  \norm{B_j(X) - B(X)}_{L^2(\Omega; L^2(0,t; V_{-2\beta}))}=0\,,\\
  &\lim_{j\to\infty}
  \norm{B_{\eps_k}(X) - B_{\eps_k}(X_j)}_{L^2(\Omega; L^2(0,t; V_{-2\beta}))}=0\,,\\
  &\lim_{j\to\infty} I_2(\eps_k, j)=0\,,
\end{align*}
we can find $j_k\in\enne$ such that 
\begin{align*}
  &\norm{B_{j_k}(X) - B(X)}_{L^2(\Omega; L^2(0,t; V_{-2\beta}))}
  +\norm{B_{\eps_k}(X) - B_{\eps_k}(X_{j_k})}_{L^2(\Omega; L^2(0,t; V_{-2\beta}))}\\
  &\qquad+|I_2(\eps_k, j_k)|\leq\frac1k\,.
\end{align*}
Hence, one has that 
\[
  |I_1(\eps_k, j_k)| + |I_2(\eps_k,j_k)|\leq \frac{C}k\,.
\]
Then, the identity \eqref{ito_uniq_aux} evaluated at $\eps_k$ and $j_k$ yields 
\begin{align*}
  &\E\left[u_{\eps_k}(X_{j_k}(t))\right] 
  - u_{\eps_k}(x_{j_k}) 
  -\lambda\E\int_0^tu_{\eps_k}(X_{j_k}(s))\,\d s
  +\E\int_0^tf_{\eps_k}(X_{j_k}(s))\,\d s\\
  &= I_1(\eps_k, j_k) + I_2(\eps_k,j_k)\,.
\end{align*}
Now, we let $k\to\infty$. To this end, note that
by the estimate \eqref{est1_eps}, for every $r\geq0$ and $\P$-a.s., we have
\begin{align*}
  |u_{\eps_k}(X_{j_k}(r))-u(X(r))|
  &\leq
  |u_{\eps_k}(X_{j_k}(r)) - u_{\eps_k}(X(r))| + |u_{\eps_k}(X(r))-u(X(r))|\\
  &\leq\norm{u_{\eps_k}}_{C^{1}_b(H)}\norm{X_{j_k}(r)-X(r)}
  +|u_{\eps_k}(X(r))-u(X(r))|\\
  &\leq C\norm{X_{j_k}(r)-X(r)}+|u_{\eps_k}(X(r))-u(X(r))|\,.
\end{align*}
Consequently, by \eqref{est1_eps}, \eqref{conv1_eps}, the properties of $P_j$,
and the dominated convergence theorem, one has, as $k\to\infty$, that 
\begin{align*}
  \E\left[u_{\eps_k}(X_{j_k}(t))\right] &\to \E\left[u(X(t))\right]\,,\\
  \E\int_0^tu_{\eps_k}(X_{j_k}(s))\,\d s &\to 
  \E\int_0^tu(X(s))\,\d s\,.
\end{align*}
Analogously, since $x_{j_k}\to x$ in $H$, by the same argument it holds also, 
as $k\to\infty$, that 
\[
  u_{\eps_k}(x_{j_k})  \to u(x)\,.
\]
Eventually, for every $r\geq0$ and $\P$-almost surely
we have that 
\begin{align*}
  |f_{\eps_k}(X_{j_k}(r))-f(X(r))|
  &\leq
  |f_{\eps_k}(X_{j_k}(r)) - f(X_{j_k}(r))| + |f(X_{j_k}(r))-f(X(r))|\\
  &\leq\norm{f_{\eps_k}-f}_{C^0_b(H)} + |f(X_{j_k}(r))-f(X(r))|\,,
\end{align*}
so that by \eqref{conv_eps_Bf}, the continuity of $f$, the properties of $P_j$,
and the dominated convergence theorem, it holds, as $k\to\infty$, that 
\[
  \E\int_0^tf_{\eps_k}(X_{j_k}(s))\,\d s\to
  \E\int_0^tf(X(s))\,\d s\,.
\]
Hence, by letting $k\to\infty$ we are left with 
\begin{align*}
  &\E\left[u(X(t))\right] 
  - u(x) 
  -\lambda\E\int_0^tu(X(s))\,\d s
  +\E\int_0^tf(X(s))\,\d s = 0 \quad\forall\,t\geq0\,.
\end{align*}
Equivalently, this can be written as
\[
  \E\left[e^{-\lambda t}u(X(t))\right] 
  - u(x) 
  +\E\int_0^te^{-\lambda s}f(X(s))\,\d s = 0 \quad\forall\,t\geq0\,,
\]
so that letting $t\to+\infty$ yields, by the dominated convergence theorem, that 
\beq
  \label{eq_uniq}
  u(x) = \int_0^\infty e^{-\lambda s}\E\left[f(X(s))\right]\,\d s \quad\forall\,f\in UC_b(H)\,.
\eeq
Analogously, the same proof applied to the solution $Y$ 
yields 
\beq
  \label{eq_uniq2}
  u(x) = \int_0^\infty e^{-\lambda s}\E\left[f(Y(s))\right]\,\d s \quad\forall\,f\in UC_b(H)\,.
\eeq
By the properties of the Laplace transform, the fact that 
any element $g\in C^0_b(H)$ can be approximated pointwise 
by a sequence in $UC_b(H)$, and by the dominated convergence theorem, 
this implies that 
\[
  \E\left[g(X(t))\right] = \E\left[g(Y(t))\right] \quad\forall\,g\in C^0_b(H)\,, 
  \quad\forall\,t\geq0\,.
\]
It follows that $X(t)$ and $Y(t)$ have the same law on $H$ for every $t\geq0$,
hence also, by pathwise continuity of $X$ and $Y$ that 
the laws of $X$ and $Y$ coincide on $C^0(\erre_+; H)$.
For a more detailed argument we refer to \cite{priola}.

\begin{remark}
Note that if $G$ is also Hilbert-Schmidt, one can handle the term $I_2$
more directly as
\begin{align*}
  |I_2(\eps,j)|&\leq\frac12\E\int_0^t\left|\operatorname{Tr}\left[(Q_j-Q)D^2
  u_\eps(X_j(s))\right]\right|\,\d s\\
  &\leq\frac{t}2\norm{Q_j-Q}_{\cL^1(H,H)}
  \norm{D^2u_\eps}_{C^0_b(H; \cL(H,H))}\\
  &\leq t\norm{G}_{\cL^2(H,H)}\norm{G_j-G}_{\cL^2(H,H)}
  \norm{D^2u_\eps}_{C^0_b(H; \cL(H,H))}\,.
\end{align*}
\end{remark}

\subsection{Proof of Theorem~\ref{thm3}}
\label{sec:uniq_unb}
We remove here 
the boundedness assumption on $B$ made in Section~\ref{sec:uniq}.
Let $x\in H$, and let 
\[
 \left(\Omega, \cF, (\cF_t)_{t\geq0}, \P, W, X,\tau_X\right)\,, \qquad
 \left(\Omega, \cF, (\cF_t)_{t\geq0}, \P, W, Y,\tau_Y\right)
\]
be two local weak solutions to \eqref{eq0} in the sense of Definition~\ref{def_sol_loc}
with respect to the same initial datum $x$, and let $\tau:=\tau_X\wedge\tau_Y$.
Let us handle the three terms appearing in \eqref{eq_sol_loc}.
Since $x\in D(A^\alpha)$ one has that 
\[
  S(\cdot)x\in C^0([0,T]; D(A^\alpha))\,, \quad\forall\,T>0\,.
\]
Moreover, since $B(X^\tau)\in L^2_{loc}(\erre_+; V_{-2\beta})$,
one has by \cite[Thms.~3.3--3.5]{vNVW}
\begin{align*}
  \int_0^{\cdot\wedge\tau} S((\cdot\wedge\tau)-s)B(X^\tau(s))\d s 
  \in C^0([0,T]; D(A^{\frac12-\beta}))\,, \quad\forall\,T>0\,,
\end{align*}
while assumption \ref{H3} and \cite[Thm.~5.15]{dapratozab} guarantee that
\begin{align*}
  \int_0^{\cdot\wedge\tau} S((\cdot\wedge\tau)-s)G\d W(s) \in 
  \bigcap_{\zeta\in[0,\xi)}C^0([0,T]; D(A^\zeta))\,, \quad\forall\,T>0\,.
\end{align*}
Hence, since $\alpha+\beta\leq\frac12$ and $\alpha<\xi$ by assumption,
one obtains that 
\[
  X^\tau,Y^\tau \in C^0([0,+\infty); D(A^\alpha))\,, \quad\P\text{-a.s.}
\]
In particular, there exists $\hat\Omega\in\cF$ with $\P(\hat\Omega)=1$
such that 
\beq\label{theta_alpha}
  X^\tau(\omega,t),Y^\tau(\omega,t)\in D(A^\alpha) \quad\forall\,t\geq0\,,
  \quad\forall\,\omega\in\hat\Omega\,.
\eeq
This allows to define, for every $N\in\enne$, the stopping times
\[
  \tau_{N}^X:=\tau\wedge\inf\left\{t\geq0: \norm{X^\tau(t)}_{2\alpha} > N
  \right\}\,, \qquad
  \tau_{N}^Y:=\tau\wedge\inf\left\{t\geq0: \norm{Y^\tau(t)}_{2\alpha} > N
  \right\}\,,
\]
with the usual convention that $\inf\emptyset=+\infty$, and we set
\[
   \tau_{N}:=\tau_{N}^X \wedge \tau_{N}^Y\,.
\]
Note that $\tau_N\nearrow\tau$ almost surely as $N\to\infty$.
Analogously, we define for every $N\in\enne$ the truncated operator
\[
  B_N:D(A^\alpha)\to D(A^{-\beta})\,, \qquad
  B_N(x):=B(\Pi_Nx)\,, \quad x\in D(A^\alpha)\,,
\]
where $\Pi_N:D(A^\alpha)\to D(A^\alpha)$ is the orthogonal projection on the closed 
convex set $\{z\in D(A^\alpha): \|z\|_{2\alpha}\leq N\}$.
Clearly, for every $N\in\enne$, 
$B_N:D(A^\alpha)\to D(A^{-\beta})$ is measurable and bounded
(since $B$ is locally bounded by assumption \ref{H2}). 
Furthermore, using the superscript to denote the stopped processes as usual,
it holds for all $t\geq0$ that
\begin{align}
  \nonumber
  X^{\tau_N}(t)&=S(t\wedge\tau_N)x
  + \int_{0}^{t\wedge\tau_N}S((t\wedge\tau_N)-s)B_N(X^{\tau_N}(s))\,\d s\\
  \label{eq_xn}
  &+\int_{0}^{t\wedge\tau_N}S((t\wedge\tau_N)-s)G\,\d W(s)\,,\\
  \nonumber
  Y^{\tau_N}(t)&=S(t\wedge\tau_N)x
  + \int_{0}^{t\wedge\tau_N}S((t\wedge\tau_N)-s)B_N(Y^{\tau_N}(s))\,\d s\\
   \label{eq_yn}
  &+\int_{0}^{t\wedge\tau_N}S((t\wedge\tau_N)-s)G\,\d W(s)\,.
\end{align}
Note that in \eqref{eq_xn}--\eqref{eq_yn} we have used 
that for every $s\leq \tau_N$ one has $B(X^{\tau}(s))=B_N(X^{\tau_N}(s))$ and
$B(Y^\tau(s))=B_N(Y^{\tau_N}(s))$.

The idea is now to construct two processes $\tilde X_N$ and $\tilde Y_N$ 
which are global
weak solutions to equation \eqref{eq0} with $B_N$ in place of $B$
and such that, setting
\[
[\![0,\tau_N]\!]:=\left\{(\omega,t)\in\Omega\times(0,+\infty): 
0\leq t\leq\tau_N(\omega)\right\}\,,
\]
it holds that $\tilde X_N\mathbbm1_{[\![0,\tau_N]\!]}
=X\mathbbm1_{[\![0,\tau_N]\!]}$
and $\tilde Y_N\mathbbm1_{[\![0,\tau_N]\!]}
=Y\mathbbm1_{[\![0,\tau_N]\!]}$ in distribution.
To this end, let us consider the equations
\begin{align*}
  \d X_N +A X_N\,\d t &= B_N(X_N)\,\d t + G\d W\,, \qquad
  X_N(0)=X(\tau_N)\,,\\
  \d Y_N +A Y_N\,\d t &= B_N(Y_N)\,\d t + G\d W\,, \qquad
  Y_N(0)=Y(\tau_N)\,.
\end{align*}
Since $X(\tau_N), Y(\tau_N)\in D(A^\alpha)$ almost surely,  as a consequence of Theorem~\ref{thm1} (with the choice $\delta_0=\alpha$)
there exist two weak solutions 
\[
  \left(\Omega_N', \cF_N', (\cF_{N,t}')_{t\geq0}, \P_N', W_N',  X_N\right)\,, \qquad
  \left(\Omega_N'', \cF_N'', (\cF_{N,t}'')_{t\geq0}, \P_N'', W_N'', Y_N\right)\,,
\]
in the sense of Definition~\ref{def_sol} (with $B_N$). At this point, we set 
\begin{alignat*}{2}
  &\tilde\Omega_N':=\Omega \times \Omega_N'\,,
  &&\tilde\Omega_N'':=\Omega \times \Omega_N''\,,\\
  &\tilde \cF_N':=\cF\otimes \cF_N'\,,\qquad
  &&\tilde \cF_N'':=\cF\otimes \cF_N''\,,\\
  &\tilde \P_N':=\P\otimes \P_N'\,, \qquad
  &&\tilde \P_N'':=\P\otimes \P_N''\,.
\end{alignat*}
As far as the Wiener process is concerned, we define
\begin{align*}
  \tilde W_N'(\omega,\omega',t)&:=W(\omega,t\wedge\tau_N)
  +W_N'(\omega',t-\tau_N(\omega))\mathbbm1_{\{t>\tau_N(\omega)\}}\,,
  \quad(\omega,\omega',t)\in\tilde\Omega_N'\times\erre_+\,,\\
  \tilde W_N''(\omega,\omega'',t)&:=W(\omega,t\wedge\tau_N)
  +W_N''(\omega'',t-\tau_N(\omega))\mathbbm1_{\{t>\tau_N(\omega)\}}\,,
  \quad(\omega,\omega'',t)\in\tilde\Omega_N''\times\erre_+\,.
\end{align*}
Analogously, we define
\begin{align*}
  \tilde X_N(\omega,\omega',t)&:=X(\omega,t)\mathbbm1_{\{t\leq\tau_N(\omega)\}}
  + X_N(\omega',t-\tau_N(\omega))\mathbbm1_{\{t>\tau_N(\omega)\}}\,,
  \quad(\omega,\omega',t)\in\tilde\Omega_N'\times\erre_+\,,\\
  \tilde Y_N(\omega,\omega'',t)&:=Y(\omega,t)\mathbbm1_{\{t\leq\tau_N(\omega)\}}
  +Y_N(\omega'',t-\tau_N(\omega))\mathbbm1_{\{t>\tau_N(\omega)\}}\,,
  \quad(\omega,\omega'',t)\in\tilde\Omega_N''\times\erre_+\,.
\end{align*}
Let us consider the filtrations
\[
  \tilde \cF_{N,t}':=\sigma(\tilde W_N'(s), \tilde X_N(s):s\leq t)\,,\quad
  \tilde \cF_{N,t}'':=\sigma(\tilde W_N''(s), \tilde Y_N(s):s\leq t)\,,\qquad t\geq0\,.
\]
Then, since $\tilde X_N=X$ and $\tilde Y_N=Y$ on $[\![0,\tau_N]\!]$
as a consequence of \eqref{eq_xn}-\eqref{eq_yn} one has
\begin{align*}
  \tilde X_N(t)&=S(t)x
  + \int_{0}^{t}S(t-s)B_N(\tilde X_N(s))\,\d s
  +\int_{0}^{t}S(t-s)G\,\d W(s)\,, \quad\forall\,t\in[0,\tau_N]\,,\\
  \tilde Y_N(t)&=S(t)x
  + \int_{0}^{t}S(t-s)B_N(\tilde Y_N(s))\,\d s
  +\int_{0}^{t}S(t-s)G\,\d W(s)\,, \quad\forall\,t\in[0,\tau_N]\,.
\end{align*}
Furthermore, for $t>\tau_N$, from the definitions
of $\tilde X_N$, $\tilde W_N'$, and from the equations
satisfied by $X$ and by $X_N$, it follows that 
\begin{align*}
  &\tilde X_N(t)=X_N(t-\tau_N)\\
  &=S(t-\tau_N)X(\tau_N) + \int_0^{t-\tau_N}S(t-\tau_N-s)B_N(X_N(s))\,\d s
  +\int_{0}^{t-\tau_N}S(t-\tau_N-s)G\,\d W_N'(s)\\
  &=S(t-\tau_N)S(\tau_N)x + 
  S(t-\tau_N)\int_{0}^{\tau_N}S(\tau_N-s)B_N(\tilde X_N(s))\,\d s\\
  &\qquad+S(t-\tau_N)\int_{0}^{\tau_N}S(\tau_N-s)G\,\d W(s)
  + \int_0^{t-\tau_N}S(t-\tau_N-s)B_N(X_N(s))\,\d s\\
  &\qquad+\int_{0}^{t-\tau_N}S(t-\tau_N-s)G\,\d W'_N(s)\\
  &=S(t)x + \int_{0}^{\tau_N}S(t-s)B_N(\tilde X_N(s))\,\d s
  +\int_{0}^{\tau_N}S(t-s)G\,\d W(s)\\
  &\qquad+\int_{\tau_N}^{t}S(t-s)B_N(X_N(s-\tau_N))\,\d s
  +\int_{\tau_N}^{t}S(t-s)G\,\d W_N'(s)\\
  &=S(t)x
  + \int_{0}^{t}S(t-s)B_N(\tilde X_N(s))\,\d s
  +\int_0^tS(t-s)G\,\d \tilde W_N'(s)\,.
\end{align*}
Arguing analogously for $\tilde Y_N$, this shows that 
the families 
\beq\label{sol_prod}
  \left(\tilde\Omega_N', \tilde\cF_N', (\tilde\cF_{N,t}')_{t\geq0}, \tilde\P_N', 
  \tilde W_N', \tilde X_N\right)\,, \qquad
  \left(\tilde\Omega_N'', \tilde\cF_N'', (\tilde\cF_{N,t}'')_{t\geq0}, 
  \tilde\P_N'', \tilde W_N'', \tilde Y_N\right)\,,
\eeq
are global weak solutions to equation \eqref{eq0} with $B_N$.

Now, we can argue exactly as in Section~\ref{ssec:uniq_bound}
by using the solution $u_N$ to the Kolmogorov equation associated to 
the bounded operator $B_N$ and some $\lambda_N>0$ large enough.
By proceeding exactly as in the proof in Section~\ref{ssec:uniq_bound}, 
one gets for all $f\in C^0_b(H)$ that
\[
  u_N(x)
  =\E\int_0^{+\infty}e^{-\lambda_N s} f(\tilde X_N(s))\,\d s
  =\E\int_0^{+\infty}e^{-\lambda_N s} f(\tilde Y_N(s))\,\d s\,.
\]
Consequently, it follows that 
$\tilde X_N$ 
and $\tilde Y_N$ have the same distribution on $C^0(\erre_+; H)$.
Recalling that $X=\tilde X_N$ and $Y=\tilde Y_N$ on $[\![0,\tau_N]\!]$,
this implies that $X\mathbbm1_{[\![0,\tau_N]\!]}$
and $X\mathbbm1_{[\![0,\tau_N]\!]}$ have the same 
distribution on $C^0(\erre_+; H)$.
Recalling that again that $\tau_N\nearrow\tau$
almost surely as $N\to\infty$, 
we conclude that $X^{\tau}$ 
and $Y^{\tau}$ have the same distribution on $C^0(\erre_+; H)$.

Finally, it remains to show that the conclusion holds also when
$G\in \cL^2(H, D(A^{\delta'}))$ for some $\delta'>0$,
and $\alpha=\frac12$ and $\beta=0$. The only difference concerns 
the estimate on the stochastic convolution.
To this end,
one has by \cite[Thms.~3.3--3.5]{vNVW} that 
\begin{align*}
  \int_0^{\cdot\wedge\tau} S((\cdot\wedge\tau)-s)G\d W(s) \in 
  \bigcap_{\theta>0}C^0([0,T]; D(A^{\frac12+\delta'-\theta}))\,, \quad\forall\,T>0\,.
\end{align*}
In the case $\alpha=\frac12$ and $\beta=0$, since $\delta'>0$
one has in particular that 
\[
  \int_0^{\cdot\wedge\tau} S((\cdot\wedge\tau)-s)G\d W(s) \in 
  C^0([0,T]; D(A^{\frac12}))\,, \quad\forall\,T>0\,,
\]
and the same argument as above allows to conclude that 
\[
  X^\tau,Y^\tau \in C^0((0,+\infty); D(A^{\frac12}))\,, \quad\P\text{-a.s.}
\]
The argument follows then as before, and this concludes the proof of Theorem~\ref{thm3}.

\subsection{Proof of Corollary~\ref{thm4}}
We set $\tilde\alpha:=\alpha+\gamma$ and 
$\tilde\beta:=\beta-\gamma$, and note that by assumption 
one has $\tilde\alpha\in[0,\xi)$ and $\tilde\beta\in[0,\frac12)$.
Moreover, we note that if $W$ is a cylindrical Wiener process on $H$, then
$\tilde W:=A^\gamma W$ is a cylindrical Wiener process on $D(A^{-\gamma})$.
Hence, the main idea is to work on the space
\[
  \tilde H:=D(A^{-\gamma})\,.
\]
By setting $\tilde e_k:=\lambda_k^\gamma e_k$ for all $k\in\enne$, 
it is immediate to check that $(\tilde e_k)_k$ is a complete orthonormal system of
$\tilde H$ such that $A\tilde e_k=\lambda_k\tilde e_k$ for all $k\in\enne$.
Consequently, one can define the operator $\tilde A$ on $\tilde H$ as the natural 
extension of $A$. More precisely, $\tilde A$ is the linear operator on $\tilde H$
with domain
\[
  D(\tilde A)=\left\{y\in \tilde H: \sum_{k=0}^\infty
  \lambda_k^2|(y,\tilde e_k)_{D(A^{-\gamma})}|^2
  <+\infty\right\}
\]
and
\[
  \tilde Ay=\sum_{k=0}^\infty\lambda_k(y, \tilde e_k)_{D(A^{-\gamma})}\tilde e_k\,, 
  \quad y\in D(\tilde A)\,.
\] 
Note that, for every $s\geq0$, a direct computation yields that 
\[
  D(\tilde A^s)= \left\{y\in \tilde H: \sum_{k=0}^\infty
  \lambda_k^{2s}|(y, \tilde e_k)_{D(A^{-\gamma})}|^2
  <+\infty\right\}
  = \left\{y\in \tilde H: \sum_{k=0}^\infty
  \lambda_k^{2s-2\gamma}|(y, e_k)_{H}|^2
  <+\infty\right\}
\]
so that
\[ 
  D(\tilde A^{s}) = D(A^{s-\gamma}) \quad\forall\,s\geq0\,,
\]
as well as $\tilde Ay=Ay$ for all $y\in D(A)$. Hence, assumption \ref{H1}
is satisfied by the operator $\tilde A$ on the space $\tilde H$.

Moreover, let us show now that assumption \ref{H3} is satisfied on the space 
$\tilde H$ with $\delta=0$ and $G=I$, i.e.~with 
$Q_t=\frac12A^{-1}(I-S(2t))$. To this end, 
we note that if $T\in\cL(\tilde H, \tilde H)$ with $T_{|H}\in\cL(H,H)$
and $T$ commutes with $A$,
then, for all $p>1$, $T_{|H}\in\cL^p(H,H)$ if and only if $T\in\cL^p(\tilde H, \tilde H)$.
Indeed, one has that 
\[
  \sum_{k=0}^\infty\|T\tilde e_k\|_{D(A^{-\gamma})}^p
  =\sum_{k=0}^\infty\|A^{-\gamma}TA^\gamma e_k\|_H^p
  =\sum_{k=0}^\infty\|Te_k\|_H^p\,.
\]
Consequently, by recalling that we use the same notation $S$
for the extension of the semigroup generated by $A$ on $\tilde H$, 
since by assumption \ref{H3} is satisfied with $\delta=0$, we have
\[
 \int_0^Tt^{-2\xi}\|S(t)\|^2_{\cL^2(\tilde H,\tilde H)}\,\d t
 =\int_0^Tt^{-2\xi}\|S(t)\|^2_{\cL^2(H,H)}\,\d t <+\infty \quad\forall\,T>0
\]
and, 
\begin{align*}
 &\int_0^{+\infty} e^{-\lambda t}
 \norm{Q_t^{-\frac12}S(t)}_{\cL^4(\tilde H,\tilde H)}^{2(1-\vartheta)}
 \norm{Q_t^{-\frac12}S(2t)}^{\vartheta}_{\cL^2(\tilde H,\tilde H)}\,\d t\\
 &\qquad=\int_0^{+\infty} e^{-\lambda t}
 \norm{Q_t^{-\frac12}S(t)}_{\cL^4(H, H)}^{2(1-\vartheta)}
 \norm{Q_t^{-\frac12}S(2t)}^{\vartheta}_{\cL^2(H, H)}\,\d t<+\infty\,.
\end{align*}
This shows then that the operator $\tilde A$ satisfies assumption \ref{H3}
on the space $\tilde H$ with $\delta=0$.

Furthermore, since one has that 
\[
D(A^{\alpha})=D(\tilde A^{\alpha+\gamma})=
D(\tilde A^{\tilde \alpha})\,, \qquad
D(A^{-\beta})=D(\tilde A^{\gamma-\beta})=D(\tilde A^{-\tilde\beta})\,,
\]
the operator $B$ is well-defined, measurable, and bounded as an operator 
\[
  B:D(\tilde A^{\tilde \alpha}) \to D(\tilde A^{-\tilde\beta})
\]
and satisfies \ref{H2} with $\tilde\alpha$ and $\tilde\beta$.

Taking all these remarks into account, equation \eqref{eq:-gamma}
can be viewed on the space $\tilde H$ in the form 
\[
  \d X + \tilde A X\,\d t = B(X)\,\d t + \d \tilde W\,.
\]
As we have checked assumptions \ref{H1}--\ref{H2}--\ref{H3}
and $B$ is bounded, the existence Theorem~\ref{thm1} ensures that for all
$x\in \tilde H$ there exists indeed a global weak solution
\[
  X\in C^0(\erre_+; \tilde H)\cap 
  L^1_{loc}(\erre_+; D(\tilde A^{\tilde\alpha}))\,, \quad\P\text{-a.s.}\,,
\]
i.e.
\[
  X\in C^0(\erre_+; D(A^{-\gamma}))\cap 
  L^1_{loc}(\erre_+; D(A^{\alpha}))\,, \quad\P\text{-a.s.}\,.
\]
Finally, uniqueness in law on $C^0(\erre_+; D(A^{-\gamma}))$
is a direct consequence of Theorem~\ref{thm2}.

\section{Examples and applications}
\label{sec:appl}

We collect here some class of equations for which we can 
prove weak uniqueness by noise as an application of 
Theorems~\ref{thm2}-\ref{thm3}.

First of all, as it was pointed out in Remark~\ref{rmk:hs},
let us state the following lemma.

\begin{lem}\label{lem:hs}
  Assume that $G\in\cL^2(H,H)$. Then, conditions \eqref{eq:cont_time}
and \eqref{eq:L4} are satisfied for every $\xi\in(0,\frac12)$ and 
  $\vartheta\in(\frac{2\delta}{2\delta+1}, 1)$
\end{lem}
\begin{proof}
  It holds that
  \[
  \int_0^Tt^{-2\xi}\|S(t)G\|^2_{\cL^2(H,H)}\,\d t
  \leq\|G\|_{\cL^2(H,H)}^2\int_0^Tt^{-2\xi}\,\d t<+\infty\quad\forall\,\xi\in(0,1/2)\,.
  \]
  Moreover, we have
\begin{align*}
  Q_t^{-\frac12}S(t)G&=\sqrt2A^{\frac12}(I-S(2t))^{-\frac12}S(t)\,,
\end{align*}
so that, for every $\eps>0$,
\begin{align*}
  \sum_{k=0}^\infty\norm{Q_t^{-\frac12}S(t)Ge_k}_{H}^4
  &=4\sum_{k=0}^\infty\lambda_k^2\frac{e^{-4t\lambda_k}}{(1-e^{-2t\lambda_k})^2}\\
  &=4\sum_{k=0}^\infty\lambda_k^2\frac1{(t\lambda_k)^{2+\eps}}
  \frac{(t\lambda_k)^{2+\eps}e^{-4t\lambda_k}}{(1-e^{-2t\lambda_k})^2}\\
  &\leq 4\left(\max_{r\geq0}\frac{r^{2+\eps}e^{-4r}}{(1-e^{-2r})^2}\right)
  \frac1{t^{2+\eps}}\sum_{k=0}^\infty\frac1{\lambda_k^\eps}\,.
\end{align*}
Recalling that $G:V_{2\delta}\to H$
is an isomorphism, one has
\begin{align*}
  \sum_{k=0}^\infty\lambda_k^{-2\delta}
  &=\sum_{k=0}^\infty\lambda_k^{-2\delta}\|e_k\|_H^2
  =\sum_{k=0}^\infty\|G^{-1}A^{-\delta}Ge_k\|_H^2\\
  &\leq\|G^{-1}\|_{\cL(V_{2\delta, H})}\|A^{-\delta}\|_{\cL(H,V_{2\delta})}
  \sum_{k=0}^\infty\|Ge_k\|_H^2<+\infty
\end{align*}
since $G\in \cL^2(H,H)$.
This shows that, for all $t>0$, $\eps\geq 2\delta$, and
$\vartheta\in(0,1)$,
\[
  Q_t^{-\frac12}S(t)G\in\cL^4(H,H) \qquad\text{and}\qquad
  \norm{Q_t^{-\frac12}S(t)G}_{\cL^4(H,H)}^{2(1-\vartheta)}
  \leq C_{\eps,\vartheta}t^{-(1+\frac\eps2)(1-\vartheta)}
\]
for some constant $C_{\eps,\vartheta}>0$ independent of $t$.
Similarly, it holds that 
\begin{align*}
  Q_t^{-\frac12}S(2t)Q&=\sqrt2A^{\frac12}(I-S(2t))^{-\frac12}S(2t)G
\end{align*}
so that, for every $\sigma>0$,
\begin{align*}
  \sum_{k=0}^\infty\norm{Q_t^{-\frac12}S(2t)Qe_k}_{H}^2
  &=2\sum_{k=0}^\infty\lambda_k
  \frac{e^{-4t\lambda_k}}{1-e^{-2t\lambda_k}}\|Ge_k\|_H^2\\
  &\leq2\left(\max_{r\geq0}\frac{re^{-4r}}{1-e^{-2r}}\right)
  \frac1{t}\|G\|^2_{\cL^2(H,H)}\,.
\end{align*}
This shows that, for all $t>0$ and
$\vartheta\in(0,1)$, one has
\[
  Q_t^{-\frac12}S(2t)Q\in\cL^2(H,H) \qquad\text{and}\qquad
  \norm{Q_t^{-\frac12}S(2t)Q}_{\cL^2(H,H)}^{\vartheta}
  \leq C_{\sigma,\vartheta}t^{-\frac\vartheta2}
\]
for some constant $C_{\sigma,\vartheta}>0$ independent of $t$.
Putting everything together, condition \eqref{eq:L4} is satisfied 
provided that $\eps\geq2\delta$ is chosen with 
$(1+\frac\eps2)(1-\vartheta)+\frac\vartheta2<1$, i.e.~$\eps(1-\vartheta)<\vartheta$.
It is easy to see that this is possible if $\vartheta\in(\frac{2\delta}{2\delta+1}, 1)$.
\end{proof}

It what follows, $\OO\subset\erre^d$ ($d\geq1$) is a smooth bounded domain,
$L^2_0(\OO)$ denotes the closed subspace of $L^2(\OO)$ of elements
with null space average,
and we use the notation 
\[
  A_D:L^2(\OO)\to L^2(\OO)\,, \qquad
  A_Dv:=-\Delta v\,, \quad v\in D(A_D):=H^2(\OO)\cap H^1_0(\OO)\,,
\]
and
\[
  A_N:L^2_0(\OO)\to L^2_0(\OO)\,, \qquad
  A_Nv:=-\Delta v\,, \quad v\in D(A_N):=\left\{v\in L^2_0(\OO)\cap H^2(\OO):
  \partial_{\bf n}v=0\right\}\,,
\]
for the Laplace operator with Dirichlet and Neumann boundary conditions, respectively.
Note that both $A_D$ and $A_N$ satisfy \ref{H1}.
Before moving to specific examples, we collect here 
some useful properties of the Laplace operator and its 
fractional powers. These will be useful 
in order to check assumptions \ref{H2}-\ref{H3}.
In particular, in the notation of the paper, 
we are interested in the choice $G=A^{-\delta}$ for some $\delta\in[0,1)$,
where $A$ is either $A_D$ or $A_N$.
\begin{prop}
  \label{prop:lap}
  Let $\delta\in[0,1)$ be fixed, 
  let $A=A_D$ or $A=A_N$, 
  let $S$ be the $C^0$-semigroup generated by $A$, 
  and let $Q_t$ be defined as in \eqref{Qt} with 
  $G=A^{-\delta}$ and $Q=A^{-2\delta}$, for $t\geq0$.
  Then, it holds that:
  \begin{enumerate}[{label=(\alph*)}]
  \item \label{it_Lap_a} $G\in\cL^2(L^2(\OO), L^2(\OO))$ 
  if and only if $\delta>\frac d4$;
  \item \label{it_Lap_a'} $G\in\cL^2(L^2(\OO), D(A^{\delta'}))$ 
  if and only if $\delta>\frac d4+\delta'$;
  \item \label{it_Lap_b} $Q_t \in \cL^1(L^2(\OO), L^2(\OO))$ if and only if 
  $\delta>\frac d4-\frac12$, for all $t>0$;
  \item \label{it_Lap_b'} if $\delta>\frac d4-\frac12$, condition \eqref{eq:cont_time}
  is satisfied for all $\xi\in(0,\delta+\frac12-\frac d4)$;
  \item \label{it_Lap_c} $Q_t^{-\frac12}S(t)G\in \cL^4(L^2(\OO), L^2(\OO))$ for all $t>0$. Moreover, 
  for all $\eps>\frac d2$ there exists a constant $C_\eps>0$ such that 
  \[
  \norm{Q_t^{-\frac12}S(t)G}_{\cL^4(L^2(\OO), L^2(\OO))}^2
  \leq C_\eps t^{-(1+\frac\eps2)} \quad\forall\,t>0\,;
  \]
  \item \label{it_Lap_d} $Q_t^{-\frac12}S(2t)Q\in \cL^2(L^2(\OO), L^2(\OO))$ for all $t>0$. Moreover, 
   for all $\sigma>\frac d2-2\delta$ there exists a constant $C_\sigma>0$ such that 
  \[
  \norm{Q_t^{-\frac12}S(2t)Q}_{\cL^2(L^2(\OO), L^2(\OO))}
  \leq C_\sigma t^{-\frac12(1+\sigma)} \quad\forall\,t>0\,;
  \]
  \item \label{it_Lap_e} if $\delta>\frac d4-\frac12$, there exists $\vartheta\in(0,1)$
  such that condition \eqref{eq:L4} is satisfied.
  \end{enumerate}
\end{prop}
As a direct consequence, we have the following corollary.
\begin{cor}
  \label{cor:lap}
  In the setting of Proposition~\ref{prop:lap}, for every $t>0$ 
  the following conditions hold.
  \begin{enumerate}[{label=(\roman*)}]
  \item If $d=1$: $G\in \cL^2(L^2(\OO), L^2(\OO))$ if and only if $\delta>\frac14$,
  $Q_t\in\cL^1(L^2(\OO), L^2(\OO))$ for every $\delta\geq0$, 
  and conditions \eqref{eq:cont_time} and \eqref{eq:L4} are satisfied 
  for every $\delta\geq0$ and $\xi\in(0,\frac14+\delta)$.
  \item If $d=2$: $G\notin \cL^2(L^2(\OO), L^2(\OO))$, 
  $Q_t\in\cL^1(L^2(\OO), L^2(\OO))$ for every $\delta>0$, 
  and conditions \eqref{eq:cont_time} and \eqref{eq:L4} are satisfied for every $\delta>0$
  and $\xi\in(0,\delta)$.
  \item If $d=3$: $G\notin \cL^2(L^2(\OO), L^2(\OO))$, 
  $Q_t\in\cL^1(L^2(\OO), L^2(\OO))$ for every $\delta>\frac14$, 
  and conditions \eqref{eq:cont_time} and \eqref{eq:L4} 
  is satisfied for every $\delta>\frac14$ and $\xi\in(0,\delta-\frac14)$.
  \end{enumerate}
\end{cor}
\begin{proof}[Proof of Proposition~\ref{prop:lap}]
  Let $(e_k)_k$ be a complete orthonormal system of $H$ 
  made of eigenvectors of $A$,
  and let $(\lambda_k)_k$ be the corresponding eigenvalues.
  Recall that $\lambda_k\sim k^{\frac 2d}$ as $k\to\infty$.
  
  {\sc Points} \ref{it_Lap_a}-\ref{it_Lap_a'}. We have that
  \[
  \sum_{k=0}^\infty\norm{Ge_k}_{D(A^{\delta'})}^2
  =\sum_{k=0}^\infty\norm{A^{\delta'}Ge_k}_{L^2(\OO)}^2
  =\sum_{k=0}^\infty\norm{A^{-(\delta-\delta')}e_k}_{L^2(\OO)}^2
  =\sum_{k=0}^\infty\lambda_k^{-2(\delta-\delta')}\,,
  \]
  where $\lambda_k^{-2(\delta-\delta')}
  \sim k^{-\frac{4}d(\delta-\delta')}$ as $k\to\infty$, so the series 
  converges if and only if $\delta-\delta'>\frac d4$.
  
  {\sc Point} \ref{it_Lap_b}. For every $t\geq0$, one has
  \[
  Q_t=\frac12 A^{-1}Q(I-S(2t))=\frac12A^{-(1+2\delta)}(I-S(2t))
  \]
  so that by \ref{it_Lap_a} we have that $Q_t\in\cL^1(L^2(\OO),L^2(\OO))$
  if and only if $1+2\delta>\frac{d}2$,
  i.e., $\delta>\frac d4-\frac12$.
  
  {\sc Point} \ref{it_Lap_b'}. For every $t>0$ and $\rho>0$ we have 
  \begin{align*}
  \|S(t)G\|^2_{\cL^2(H,H)}&=\sum_{k=0}^\infty\|S(t)A^{-\delta}e_k\|_H^2
  =\sum_{k=0}^\infty e^{-2t\lambda_k}\lambda_k^{-2\delta}
  =\frac1{t^\rho}\sum_{k=0}^\infty \frac{e^{-2t\lambda_k}(t\lambda_k)^{\rho}}{\lambda_k^{2\delta+\rho}}\\
  &\leq \max_{r\geq0}\left(e^{-2r}r^\rho\right)\frac1{t^\rho}
  \sum_{k=0}^\infty\frac1{\lambda_k^{2\delta+\rho}}\,.
  \end{align*}
  Since $\lambda_k\sim k^{\frac2d}$ as $k\to\infty$, the series on the right-hand side
  converges if and only if $\frac2d(2\delta+\rho)>1$, hence $\rho>\frac{d}2-2\delta$.
  Now, condition \eqref{eq:cont_time} is satisfied for some $\xi\in(0,\frac12)$
  provided that $2\xi+\rho<1$. These two conditions are possible if and only if
  \[
  \frac d2-2\delta<1-2\xi\,.
  \]
  Consequently, if $\delta>\frac d4-\frac12$, there exists $\xi\in(0,\frac12)$
  such that $\delta>\frac d4+\xi-\frac12$ (e.g.~$\xi\in(0,\delta+\frac12-\frac d4)$):
  hence, it holds also that $\frac d2-2\delta<1-2\xi$, and
  one can choose $\rho \in (\frac d2-2\delta, 1-2\xi)$. The computations above imply then 
  that condition \eqref{eq:cont_time} is satisfied.
  
  {\sc Point} \ref{it_Lap_c}. We have
\begin{align*}
  Q_t^{-\frac12}S(t)G&=\sqrt2A^{\frac12+\delta}(I-S(2t))^{-\frac12}S(t)A^{-\delta}\\
  &=\sqrt2A^{\frac12}(I-S(2t))^{-\frac12}S(t)\,,
\end{align*}
so that, for every $\eps>0$,
\begin{align*}
  \sum_{k=0}^\infty\norm{Q_t^{-\frac12}S(t)Ge_k}_{L^2(\OO)}^4
  &=4\sum_{k=0}^\infty\lambda_k^2\frac{e^{-4t\lambda_k}}{(1-e^{-2t\lambda_k})^2}\\
  &=4\sum_{k=0}^\infty\lambda_k^2\frac1{(t\lambda_k)^{2+\eps}}
  \frac{(t\lambda_k)^{2+\eps}e^{-4t\lambda_k}}{(1-e^{-2t\lambda_k})^2}\\
  &\leq 4\left(\max_{r\geq0}\frac{r^{2+\eps}e^{-4r}}{(1-e^{-2r})^2}\right)
  \frac1{t^{2+\eps}}\sum_{k=0}^\infty\frac1{\lambda_k^\eps}\,.
\end{align*}
Recalling that $\lambda_k\sim k^{\frac2d}$ as $k\to\infty$,
the series on the right-hand side converges if and only if $\frac{2\eps}d>1$,
i.e., $\eps>\frac d2$. This shows that, for all $t\geq0$, $\eps>\frac d2$, and
$\vartheta>0$,
\[
  Q_t^{-\frac12}S(t)G\in\cL^4(L^2(\OO),L^2(\OO)) \qquad\text{and}\qquad
  \norm{Q_t^{-\frac12}S(t)G}_{\cL^4(L^2(\OO),L^2(\OO))}^{2(1-\vartheta)}
  \leq C_{\eps,\vartheta}t^{-(1+\frac\eps2)(1-\vartheta)}
\]
for some constant $C_{\eps,\vartheta}>0$ independent of $t$.

{\sc Point} \ref{it_Lap_d}. It holds that 
\begin{align*}
  Q_t^{-\frac12}S(2t)Q&=\sqrt2A^{\frac12+\delta}(I-S(2t))^{-\frac12}S(2t)A^{-2\delta}\\
  &=\sqrt2A^{\frac12-\delta}(I-S(2t))^{-\frac12}S(2t)\,,
\end{align*}
so that, for every $\sigma>0$,
\begin{align*}
  \sum_{k=0}^\infty\norm{Q_t^{-\frac12}S(2t)Qe_k}_{L^2(\OO)}^2
  &=2\sum_{k=0}^\infty\lambda_k^{1-2\delta}
  \frac{e^{-4t\lambda_k}}{1-e^{-2t\lambda_k}}\\
  &=2\sum_{k=0}^\infty\lambda_k^{1-2\delta}
  \frac1{(t\lambda_k)^{1+\sigma}}
  \frac{(t\lambda_k)^{1+\sigma}e^{-4t\lambda_k}}{1-e^{-2t\lambda_k}}\\
  &\leq 2\left(\max_{r\geq0}\frac{r^{1+\sigma}e^{-4r}}{1-e^{-2r}}\right)
  \frac1{t^{1+\sigma}}\sum_{k=0}^\infty\frac1{\lambda_k^{2\delta+\sigma}}\,.
\end{align*}
Recalling again that $\lambda_k\sim k^{\frac2d}$ as $k\to\infty$,
the series on the right-hand side converges if and only if $\frac2d(2\delta+\sigma)>1$,
i.e., $\sigma>\frac d2-2\delta$. 
This shows that, for all $t\geq0$, $\sigma>\frac d2-2\delta$, and
$\vartheta>0$,
\[
  Q_t^{-\frac12}S(2t)Q\in\cL^2(L^2(\OO),L^2(\OO)) \qquad\text{and}\qquad
  \norm{Q_t^{-\frac12}S(2t)Q}_{\cL^2(L^2(\OO),L^2(\OO))}^{\vartheta}
  \leq C_{\sigma,\vartheta}t^{-\frac\vartheta2(1+\sigma)}
\]
for some constant $C_{\sigma,\vartheta}>0$ independent of $t$.

{\sc Point} \ref{it_Lap_e}. Putting everything together, one has for every $\lambda>0$,
$\eps>\frac d2$, and $\sigma>\frac d2-2\delta$ that 
\begin{align*}
  &\int_0^{+\infty} e^{-\lambda t}
 \norm{Q_t^{-\frac12}S(t)G}_{\cL^4(L^2(\OO),L^2(\OO))}^{2(1-\vartheta)}
 \norm{Q_t^{-\frac12}S(2t)Q}^{\vartheta}_{\cL^2(L^2(\OO),L^2(\OO))}\,\d t\\
 &\qquad\leq C_{\eps,\vartheta}C_{\sigma,\vartheta}
 \int_{0}^{+\infty}e^{-\lambda t} \frac1{t^{(1+\frac\eps2)(1-\vartheta) 
 + \frac\vartheta2(1+\sigma)}}\,\d t\,.
\end{align*}
The integral on the right-hand side converges if and only if 
$(1+\frac\eps2)(1-\vartheta) 
 + \frac\vartheta2(1+\sigma)
 <1$, i.e., $\eps(1-\vartheta) + \sigma\vartheta<\vartheta$.
 In order to guarantee also the conditions 
 $\eps>\frac d2$ and $\sigma>\frac d2-2\delta$,
 $\delta$ and $\vartheta$ have to be chosen such that 
\[
  \frac d2(1-\vartheta) + \vartheta\left(\frac d2 -2\delta\right)<\vartheta\,.
\]
Direct computations show that this is equivalent to 
\[
  \vartheta>\frac{d}{2(1+2\delta)}\,.
\]
We can choose then $\vartheta\in(0,1)$ satisfying such inequality 
if and only if 
\[
\frac{d}{2(1+2\delta)} <1\,,
\]
i.e., $\delta>\frac d4 - \frac12$, as required.
\end{proof}

\subsection{Heat equations with perturbation}
We consider SPDEs in the form
  \begin{align*}
  \d X - \Delta X\,\d t = F(X)\,\d t + G\,\d W \quad&\text{in } (0,T)\times\OO\,,\\
  X=0\quad&\text{in } (0,T)\times\partial\OO\,,\\
  X(0)=x_0 \quad&\text{in } \OO\,,
  \end{align*}
where $\OO\subset\erre^d$ is a smooth bounded domain, 
$\Delta$ and $\nabla$ denote the Laplacian and the gradient with 
respect to the space-variables, 
$F:\erre\to\erre$ is continuous and linearly bounded, 
$x_0\in L^2(\OO)$, and $G\in \cL(L^2(\OO),L^2(\OO))$.
We define the Nemytskii-type operator 
\[
  B:L^2(\OO)\to L^2(\OO)\,, \qquad
  B(\varphi)(\xi):=F(\varphi(\xi))\,,
  \quad\xi\in\OO\,,\quad\varphi\in L^2(\OO)\,,
\]
which is measurable, linearly bounded, and strongly continuous 
(by continuity and linear boundedness of $F$). 
Then, setting $H:=L^2(\OO)$ and $G:=A_D^{-\delta}$ for some $\delta\in[0,1)$,
the abstract formulation of the problem reads
as the following evolution equation on $H$:
\[
  \d X + A_D X\,\d t = B(X)\,\d t + G\,\d W\,, \quad X(0)=x_0\,.
\]
In the notation of the paper, we have then $\alpha=\beta=0$. 
By Corollary~\ref{cor:lap}
we deduce that uniqueness in distribution holds up to
dimension $d=3$ included, with the choices
$\delta\in[0,\frac12)$ in dimension $d=1$,
$\delta\in(0,\frac12)$ in dimension $d=2$, and 
$\delta\in(\frac14,\frac12)$ in dimension $d=3$.

\subsection{Heat equations with polynomial perturbation}
We consider SPDEs in the form
  \begin{align*}
  \d X - \Delta X\,\d t  = 
  F(X)\,\d t + G\,\d W \quad&\text{in } (0,T)\times\OO\,,\\
  X=0\quad&\text{in } (0,T)\times\partial\OO\,,\\
  X(0)=x_0 \quad&\text{in } \OO\,,
  \end{align*}
where $\OO\subset\erre^d$ is a smooth bounded domain,
$F:\erre\to\erre$ is a continuous function,
$x_0\in L^2(\OO)$, and $G\in \cL(L^2(\OO),L^2(\OO))$.
We further suppose that 
$F$ behaves like a polynomial, in the sense that
there exist $p\geq2$ and $C>0$ such that 
\[
  |F(r)|\leq C(1+|r|^{p-1})\quad\forall\,r\in\erre\,.
\]
Given $r\in[\max\{2,p-1\}, 2(p-1)]$, setting $s:=\frac{r}{p-1}\in[1,2]$
one can view $B$ as an operator 
\[
  B:L^r(\OO)\to L^s(\OO)
\]
as
\begin{align*}
  B(\varphi)(\xi)&:=
  F(\varphi(\xi))\,,
  \quad\varphi\in L^r(\OO)\,,\quad\xi\in\OO\,.
\end{align*}
The idea is to choose $\alpha\in[0,\frac12]$ and $\beta\in[0,\frac12)$
such that $\alpha+\beta\in[0,\frac12]$, $D(A_D^\alpha)\embed L^r(\OO)$,
and $L^s(\OO)\embed D(A_D^{-\beta})$, so that
$B$ can be considered as
\[
  B:D(A_D^\alpha)\to D(A_D^{-\beta})
\]
given by
\[
  \ip{B(\varphi)}{\psi}_{V_{-2\beta}, V_{2\beta}}:=
  \int_\OO F(\varphi(\xi))
  \psi(\xi)\,\d\xi\,,
  \quad\varphi\in D(A_D^\alpha)\,,\psi\in D(A_D^{\beta})\,.
\]
Note that such a choice yields
a well-defined, measurable, and strongly-weakly continuous operator
(by continuity and polynomial growth of $F$).
Then, setting $H:=L^2(\OO)$ and $G:=A_D^{-\delta}$ for some $\delta\in[0,1)$,
the abstract formulation of the problem reads
as the following evolution equation on $H$:
\[
  \d X + A_D X\,\d t = B(X)\,\d t + G\,\d W\,, \quad X(0)=x_0\,.
\]
Global existence of solutions follows by classical variational arguments 
along with coercivity-type conditions in the form 
\[
  -F(r)r\geq C^{-1}|r|^p - C \quad\forall\,r\in\erre\,.
\]
If coercivity fails, or more generally if $F$ is not
necessarily nonincreasing nor polynomially-bounded, then 
only local existence can be proved.

As far as uniqueness is concerned, 
we have to carefully tune the parameters $\alpha$ and $\beta$
in order to achieve the best possible uniqueness result.
To this end, 
by the classical Sobolev embedding, the inclusions
$D(A_D^\alpha)\embed L^r(\OO)$
and $L^s(\OO)\embed D(A_D^{-\beta})$
are satisfied if
$\alpha=\frac{d}{2}(\frac12-\frac1r)$ and 
$\beta=\frac{d}2(\frac1s-\frac12)$. With such a choice, 
it is immediate to see that $\alpha\geq0$ and $\beta\geq0$.
Moreover, note that $\alpha+\beta
=\frac{d}{2}(\frac1s-\frac1r)=\frac{d(p-2)}{2r}$, so that 
$\alpha+\beta\leq\frac12$ provided that 
$r\geq d(p-2)$.
Hence, it is possible to choose $r$ such that 
$r\in[\max\{2,p-1\}, 2(p-1)]$ and $r\geq d(p-2)$
if and only if $d(p-2)\leq 2(p-1)$.
In dimensions $d=1,2$ this is always true, while in dimension 
$d=3$ this requires the limitation $p\leq4$.
Clearly, the optimal choice of $r$ is the one that minimises $\alpha$.
By definition of $\alpha$, this is the minimal admissible value of $r$,
i.e.~$r_{opt}=\max\{2,p-1,d(p-2)\}$ (with the constraint $p\leq4$ in $d=3$).
To summarise, we have then the following uniqueness results.

{\sc Dimension} $d=1$.
We can handle every growth $p\in[2,+\infty)$.
The optimal choice of $r$ is 
$r_{opt}=\max\{2,p-1,p-2\}=\max\{2,p-1\}$, yielding 
$\alpha_{opt}=\frac12(\frac12-\frac1{\max\{2,p-1\}})$.
Thanks to Corollary~\ref{cor:lap},
we can apply Theorem~\ref{thm3} and infer that weak uniqueness holds
for all $\delta\in[0,\frac12)$ and initial datum $x_0\in D(A^{\alpha_{opt}})$.

{\sc Dimension} $d=2$.
We can handle every growth $p\in[2,+\infty)$.
The optimal choice of $r$ is 
$r_{opt}=\max\{2,p-1,2(p-2)\}$, yielding 
$\alpha_{opt}=\frac12-\frac1{\max\{2,p-1, 2(p-2)\}}$.
Thanks to Corollary~\ref{cor:lap},
we can apply Theorem~\ref{thm3} and infer that weak uniqueness holds
for all $\delta\in(\alpha_{opt},\frac12)$
and initial datum $x_0\in D(A^{\alpha_{opt}})$.

{\sc Dimension} $d=3$.
We can handle every growth $p\in[2,4)$.
The optimal choice of $r$ is 
$r_{opt}=\max\{2,p-1,3(p-2)\}$, yielding 
$\alpha_{opt}=\frac32(\frac12-\frac1{\max\{2,p-1, 3(p-2)\}})$.
Note that $\alpha_{opt}\in[0,\frac12)$ since $p\in[2,4)$.
Hence, thanks to Corollary~\ref{cor:lap},
we can apply Theorem~\ref{thm3} and infer that weak uniqueness holds
for all $\delta\in(\alpha_{opt},\frac12)$
and initial datum $x_0\in D(A^{\alpha_{opt}})$..
Let us point out that here the case $p=4$ is excluded as it leads
to the choice $\alpha_{opt}=\frac12$. In particular, this would require a Hilbert-Schmidt condition on $G$ in Theorem~\ref{thm3}, which is not true in dimension $d=3$.

\subsection{Heat equations with divergence-like perturbation: 
the sub-critical case}
We consider SPDEs in the form
  \begin{align*}
  \d X - \Delta X\,\d t = (-\Delta)^\beta 
  F(X)\,\d t + G\,\d W \quad&\text{in } (0,T)\times\OO\,,\\
  X=0\quad&\text{in } (0,T)\times\partial\OO\,,\\
  X(0)=x_0 \quad&\text{in } \OO\,,
  \end{align*}
where $\OO\subset\erre^d$ is a smooth bounded domain, $\beta\in(0,\frac12)$,
$F:\erre\to\erre$ is continuous and linearly bounded, 
$x_0\in L^2(\OO)$, and $G\in \cL(L^2(\OO),L^2(\OO))$.
We define the operator 
\[
  B:L^2(\OO)\to D(A_D^{-\beta})
\]
as
\[
  \ip{B(\varphi)}{\psi}_{V_{-2\beta}, V_{2\beta}}:=
  \int_\OO F(\varphi(\xi))(A_D^\beta \psi)(\xi)\,\d\xi\,,
  \quad\varphi\in L^2(\OO)\,,\quad \psi\in D(A_D^{\beta})\,,
\]
which is measurable, linearly bounded, and strongly continuous 
(by continuity and linear boundedness of $F$). 
Then, setting $H:=L^2(\OO)$ and $G:=A_D^{-\delta}$ for some $\delta\in[0,1)$,
the abstract formulation of the problem reads
as the following evolution equation on $H$:
\[
  \d X + A_D X\,\d t = B(X)\,\d t + G\,\d W\,, \quad X(0)=x_0\,.
\]
In the notation of the paper, we have then $\alpha=0$ and $\beta>0$. Now,
thanks to Corollary~\ref{cor:lap}
we deduce that uniqueness in distribution holds up to
dimension $d=3$ included, with the choices:
\begin{itemize}
\item dimension $d=1$: 
$\beta\in(0,\frac12)$ and $\delta\in[0,\frac12-\beta)$,
\item dimension $d=2$: 
$\beta\in(0,\frac12)$ and $\delta\in(0,\frac12-\beta)$,
\item dimension $d=3$: 
$\beta\in(0,\frac14)$ and $\delta\in(\frac14,\frac12-\beta)$.
\end{itemize}

\subsection{Heat equations with divergence-like perturbation: 
the super-critical case}
We consider SPDEs in the form
  \begin{align*}
  \d X - \Delta X\,\d t = (-\Delta)^\beta 
  F(X)\,\d t + (-\Delta)^\gamma\,\d W \quad&\text{in } (0,T)\times\OO\,,\\
  X=0\quad&\text{in } (0,T)\times\partial\OO\,,\\
  X(0)=x_0 \quad&\text{in } \OO\,,
  \end{align*}
where $\OO\subset\erre$ is a smooth bounded domain, $\beta\in[\frac12,\frac34)$,
$F:\erre\to\erre$ is continuous and bounded, 
$\gamma\in(\beta-\frac12,\frac14)$ and $x_0\in D(A_D^{-\gamma})$.
We define the operator $B:L^2(\OO)\to D(A_D^{-\beta})$
as in the previous example, i.e.
\[
  \ip{B(\varphi)}{\psi}_{V_{-2\beta}, V_{2\beta}}:=
  \int_\OO F(\varphi(\xi))(A_D^\beta \psi)(\xi)\,\d\xi\,,
  \quad\varphi\in L^2(\OO)\,,\quad \psi\in D(A_D^{\beta})\,,
\]
which is measurable, bounded, and strongly continuous.
We note that thanks to Corollary~\ref{cor:lap} assumption 
\ref{H3} is satisfied with $\delta=0$ and $\xi\in(0,\frac14)$.
Hence,
the hypotheses of Corollary~\ref{thm4} are met and 
we deduce that the equation 
\[
  \d X + A_D X\,\d t = B(X)\,\d t + A^\gamma\,\d W\,, \quad X(0)=x_0
\]
admits a global weak solution which is unique in distribution,
in the sense of Corollary~\ref{thm4}.

\subsection{Heat equations with  non-divergence-like perturbation}
We consider SPDEs in the form
  \begin{align*}
  \d X - \Delta X\,\d t = 
  F((-\Delta)^\alpha X)\,\d t + G\,\d W \quad&\text{in } (0,T)\times\OO\,,\\
  X=0\quad&\text{in } (0,T)\times\partial\OO\,,\\
  X(0)=x_0 \quad&\text{in } \OO\,,
  \end{align*}
where $\OO\subset\erre^d$ is a smooth bounded domain, $\alpha\in(0,1)$,
$F:\erre\to\erre$ is continuous and linearly bounded, 
$x_0\in L^2(\OO)$, and $G\in \cL(L^2(\OO),L^2(\OO))$.
We define the operator 
\[
  B:D(A_D^{\alpha})\to L^2(\OO)
\]
as
\[
  B(\varphi)(\xi):=
  F((A_D^\alpha \varphi)(\xi))\,,
  \quad\varphi\in D(A_D^{\alpha})\,,
\]
which is measurable, linearly bounded, and strongly continuous 
(by continuity and linear boundedness of $F$). 
Then, setting $H:=L^2(\OO)$ and $G:=A_D^{-\delta}$ for some $\delta\in[0,1)$,
the abstract formulation of the problem reads
as the following evolution equation on $H$:
\[
  \d X + A_DX\,\d t = B(X)\,\d t + G\,\d W\,, \quad X(0)=x_0\,.
\]
In the notation of the paper, we have then $\alpha>0$ and $\beta=0$. 
We note straightaway that if $F$ is bounded,
thanks to Corollary~\ref{cor:lap}, 
we can apply Theorem~\ref{thm2}
and deduce that uniqueness in distribution holds 
for every $\alpha\in(0,1)$, $\delta\in[0,\frac12)$,
and initial datum in $H$, in every dimension $d=1,2,3$.
Otherwise, if $F$ is unbounded, depending on the dimension, we can proceed as follows.

{\sc Dimension} $d=1$.
Thanks to Corollary~\ref{cor:lap}, we can apply Theorem~\ref{thm3}
to deduce that uniqueness in distribution holds 
for every $\alpha\in(0,\frac12)$, $\delta\in[0,\frac12)$,
and initial datum in $D(A^\alpha)$.
Note also that if $F$ is unbounded and $\alpha=\frac12$,
we can still infer uniqueness in distribution
by the last statement of Theorem~\ref{thm3} 
for every $\delta\in(\frac14,\frac12)$.

{\sc Dimension} $d=2$.
Thanks to Corollary~\ref{cor:lap}, 
we can apply Theorem~\ref{thm3}
to deduce that uniqueness in distribution holds 
for every $\alpha\in(0,\frac12)$, $\delta\in(\alpha,\frac12)$,
and initial datum in $D(A^\alpha)$.

{\sc Dimension} $d=3$.
Thanks to Corollary~\ref{cor:lap}, 
we can apply Theorem~\ref{thm3}
to deduce that uniqueness in distribution holds 
for every $\alpha\in(0,\frac14)$, $\delta\in(\alpha+\frac14,\frac12)$,
and initial datum in $D(A^\alpha)$.

\begin{remark}
  Let us point put that a combination of the techniques used in 
  the previous sections allows to consider also equations with
  perturbation in mixed form, such as
  We consider SPDEs in the form
  \begin{align*}
  \d X - \Delta X\,\d t = 
  (-\Delta)^\beta F((-\Delta)^\alpha X)\,\d t + G\,\d W \quad&\text{in } (0,T)\times\OO\,,\\
  X=0\quad&\text{in } (0,T)\times\partial\OO\,,\\
  X(0)=x_0 \quad&\text{in } \OO\,,
  \end{align*}
  where either $G\in\cL(L^2(\OO), L^2(\OO))$ in the subcritical case $\beta\in[0,\frac12)$,
  or $G=A_D^\gamma$ with $\gamma>0$ in the supercritical case $\beta\in[\frac12,\frac34)$
  in dimension $d=1$.
\end{remark}

\subsection{One-dimensional Burgers equation with perturabation}
We consider SPDEs in the form
  \begin{align*}
  \d X - \Delta X\,\d t = X\nabla X\,\d t + F(X,\nabla X)\,\d t 
  + G\,\d W \quad&\text{in } (0,T)\times\OO\,,\\
  X=0\quad&\text{in } (0,T)\times\partial\OO\,,\\
  X(0)=x_0 \quad&\text{in } \OO\,,
  \end{align*}
where $\OO\subset\erre$ is a smooth bounded domain, 
$F:\erre^2\to \erre$ is measurable and bounded, 
$x_0\in H^1_0(\OO)$, and $G\in \cL^2(L^2(\OO),L^2(\OO))$.
We define the operator 
\[
  B:D(A_D^{\frac12})\to L^2(\OO)\,, \qquad
  B(\varphi)(\xi):=\varphi(\xi)\nabla\varphi(\xi) + 
  F(\varphi(\xi), \nabla\varphi(\xi))\,,
  \quad\xi\in\OO\,,\quad\varphi\in D(A_D^{\frac12})\,,
\]
which is well-defined, measurable, and strongly continuous 
since $H^1_0(\OO)\embed L^\infty(\OO)$ is continuous in dimension $1$. 
Moreover, for every $\delta\in(\frac14,\frac12)$
and for every $\delta'\in(0,\delta-\frac14)
\subset(0,\delta)\subset(0,\frac12)$,
we have that that 
$G:=A_D^{-\delta}\in \cL^2(H,D(A^{\delta'}))$
and the abstract formulation of the problem reads
\[
  \d X + A_D X\,\d t = B(X)\,\d t + G\,\d W\,, \quad X(0)=x_0\,.
\]
Hence, Theorem~\ref{thm3} applies with the choices 
$\alpha=\frac12$ and $\beta=0$, 
so that uniqueness in distribution holds for initial data in $H^1_0(\OO)$.
Note that in the case $F=0$ (i.e.~the classical Burgers equation)
also strong uniqueness is guaranteed (see e.g.~\cite{gy98})
by using direct computational techniques on the SPDE.
In this case, our approach only provides an alternative proof 
of weak uniqueness.
Existence of solutions for $F=0$ is also well-understood,  
even in higher dimensions (see for example \cite[Ex.~5.1.8]{LiuRo}).

\subsection{Cahn--Hilliard equations with perturbation}
We consider SPDEs in the form
  \begin{align*}
  \d X -\Delta (-\Delta X +F_1(X))\,\d t 
  = F_2(X, \nabla X, D^2X)\,\d t + G\,\d W \quad&\text{in } (0,T)\times\OO\,,\\
  \partial_{\bf n}X=\partial_{\bf n}\Delta X=0
  \quad&\text{in } (0,T)\times\partial\OO\,,\\
  X(0)=x_0 \quad&\text{in } \OO\,,
  \end{align*}
where $\OO\subset\erre^d$ is a smooth bounded domain with $d\in\{1,2,3\}$,
and $F_1:\erre\to\erre$ is of class $C^2$
with locally Lipschitz-continuous derivatives $F_1'$, $F_1''$.
Typically, $F_1$ is the derivative of
a double-well potential: the classical example is given 
by the quartic double-well potential
\begin{equation}
\label{eq:doubleWellPot}
  F_1(r)=\frac{\d}{\d r}\left[\frac14(r^2-1)^2\right]=r(r^2-1)\,, \quad r\in\erre\,.
\end{equation}
We stress that in our framework no assumption on 
the growth of the nonlinear term $F_1$ is required:
for example, we can cover polynomial potentials of every order,
as well as exponential-growth potentials.
The term $F_2:\erre\times\erre^d\times\erre^{d\times d}\to\erre$ 
is only continuous and linearly bounded in every argument,
representing some singular source term possibly depending 
on all derivatives up to second order, 
$x_0\in L^2(\OO)$, and $G\in \cL(L^2(\OO),L^2(\OO))$.
Noting that
\[
  -\Delta F_1(\varphi)=-\operatorname{div}[F_1'(\varphi)\nabla\varphi]
  =-F_1''(\varphi)|\nabla\varphi|^2 -F_1'(\varphi)\Delta\varphi
  \quad\forall\,\varphi\in H^2(\OO)\,,
\]
we define the operator 
\[
  B_1:H^2(\OO)\to L^2(\OO)
\]
as
\[
  B_1(\varphi)(\xi):=-F_1''(\varphi(\xi))|\nabla\varphi(\xi)|^2 
  -F_1'(\varphi(\xi))\Delta\varphi(\xi)\,,
  \quad\varphi\in H^2(\OO)\,, \quad\xi\in\OO\,,
\]
which is well-defined, measurable, locally bounded 
(by the inclusion $H^2(\OO)\embed L^\infty(\OO)$
and continuity of $F_1''$ and $F_1'$), 
and strongly continuous 
(by local Lipschitz-continuity of $F_1'$ and $F_1''$). 
Analogously, we define
\[
  B_2:H^2(\OO)\to L^2(\OO)
\]
as
\[
  B_2(\varphi)(\xi):=F_2(\varphi(\xi), \nabla \varphi(\xi), D^2\varphi(\xi))\,, 
  \quad\varphi\in L^2(\OO)\,, \quad\xi\in\OO\,,
\]
which is measurable, locally bounded, and 
strongly continuous by the hypothesis on $F_2$.
Then, setting $H:=L^2(\OO)$ and noting that 
$D((A_N^2)^{\frac12})=D(A_N)
=\{\varphi\in H^2(\OO):\partial_{\bf n}\varphi=0\}$, 
we consider the operator 
\[
  B:D(A_N)\to H\,, \qquad B(\varphi)=B_2(\varphi)-B_1(\varphi)\,, 
  \quad \varphi\in D(A_N)\,.
\]
With this notation, by setting $G:=(A_N^2)^{-\delta}=A_N^{-2\delta}$ 
for some $\delta\in[0,1)$,
the abstract formulation of the problem reads
\[
  \d X + A_N^2 X\,\d t = B(X)\,\d t + G\,\d W\,, \quad X(0)=x_0\,.
\]
In the notation of the paper, we have then $\alpha=\frac12$, $\beta=0$, 
and $A=A_N^2$.

{\sc Dimension} $d=1,2,3$.
For every $\delta\in(\frac{d}8,\frac12)$ 
and 
for every $\delta'\in(0,\delta-\frac{d}8)
\subset(0,\delta)\subset(0,\frac12)$,
by standard computations it holds that 
$G\in \cL^2(H,D(A^{\delta'}))$. 
Hence, assumption \ref{H3} is satisfied,
Theorem~\ref{thm3} applies
and uniqueness in
distribution holds with such choice of $G$ and $\beta$ for every 
initial data $x\in D(A_N)$.

\appendix

\section{Existence of weak solutions}
\label{sec:ex}
In this appendix, we provide a weak existence result for equation \eqref{eq0}
for the case where $B$ is bounded. This result follows essentially
 by a stochastic maximal regularity argument.

We emphasise that here we do not aim for optimality. 
This is just a sufficient condition for existence in a general setting. 
More refined results can be found in the literature 
according to the specific example of SPDE in consideration.
In general, if $B$ is unbounded, local existence results can be 
obtained with classical techniques in several cases, e.g.~under 
locally Lipschitz-continuity assumptions or the 
so-called structure condition between $B$ and $G$ (see \cite{AMP2023}).

Let us stress that
even if $B$ is bounded, existence of weak solution is not a direct
consequence of Girsanov theorem due to the presence of 
a coloured noise (which is indeed necessary to recover the $V_{2\alpha}$-regularity).
In the following result we give a sufficient condition for global existence 
in the case of a bounded $B$.

\begin{thm}
  \label{thm1}
  Assume \ref{H1}-\ref{H2}-\ref{H3}, and that
  there exists a Hilbert space $Z\supseteq D(A^{-\beta})$ such that
  $B$ is bounded from $D(A^\alpha)$ to $D(A^{-\beta})$
  and strongly-weakly continuous from $D(A^\alpha)$ to $Z$.
  Let $\delta_0\in[0,1]$ and $x\in D(A^{\delta_0})$.
  Then: 
  \begin{enumerate}[leftmargin=*]
  \item[(i)] if $\alpha<\xi$, 
  there exists a weak solution to \eqref{eq0} 
  in the sense of Definition~\ref{def_sol}
  such that, for all $\zeta\in[0,\xi)$,
  \begin{align*}
  X  \in C^0(\erre_+; D(A^{\delta_0\wedge\zeta}))\cap 
  L^1_{loc}(\erre_+; D(A^\alpha))\,,\qquad
  B(X) \in L^\infty(\erre_+; D(A^{-\beta}))\,, \qquad \P\text{-a.s.}
  \end{align*}
  \item[(ii)] if
  $G\in \cL^2(H, D(A^{\delta'}))$ for some $\delta'\in(0,\delta)$
  and $\alpha<(\delta_0+1)\wedge(\delta'+\frac12)$,
  there exists a weak solution to \eqref{eq0} 
  in the sense of Definition~\ref{def_sol}
  such that
  \begin{align*}
  X  \in C^0(\erre_+; D(A^{\delta_0\wedge\delta'}))\cap 
  L^1_{loc}(\erre_+; D(A^\alpha))\,,\qquad
  B(X) \in L^\infty(\erre_+; D(A^{-\beta}))\,, \qquad \P\text{-a.s.}
  \end{align*}
  \end{enumerate}
\end{thm}

\begin{proof}
We recall that $(e_k)_{k\in\enne}$ is a complete orthonormal system of $H$
made of eigenvectors of $A$, and $(\lambda_k)_{k\in\enne}$ are the corresponding 
eigenvalues. We define $H_n:=\operatorname{span}\{e_1,\ldots,e_n\}$
and we introduce $P_n:H\to H_n$ 
as the orthogonal projection on $H_n$, for all $n\in\enne$. 
Note that $P_n$ can be extended to a linear continuous operator 
$P_n:D(A^{-\beta})\to H_n$ as follows
\[
  P_nv:=\sum_{k=1}^n\langle v, e_k\rangle e_k\,, \quad v\in D(A^{-\beta})\,.
\]
We consider the approximated problem 
\beq
  \label{eq_app}
   \d X_n + AX_n\,\d t = B_n(X_n)\,\d t +  G_n\,\d W_n\,, \qquad X(0)=x_n\,,
\eeq
where 
\begin{equation*}
B_n:=P_n\circ B\circ P_n:H\to H_n\,, \qquad G_n:=P_n\circ G\in\cL^2(H,H_n)\,,
\end{equation*}
$W_n$ is an $n$-dimensional standard Brownian motion, 
and $x_n:=P_nx\in H_n$.
If we look for a solution $X_n$ in the form 
\[
  X_n(t)=\sum_{k=1}^na_k(t)e_k\,, \quad t\geq0\,,
\]
plugging it into equation \eqref{eq_app} and testing by arbitrary elements in 
$\{e_1,\ldots,e_n\}$
yields the following system of SDEs for 
the coefficients $\{a_k\}_{k=1}^n$:
\[
  \d a_j + \lambda_ja_j\,\d t = \left(B_n\left(\sum_{k=1}^na_ke_k\right), e_j\right) 
  + (G_n\,\d W_n, e_j)\,, \quad a_j(0)=(x_n,e_j)\,,\qquad j=1,\ldots,n\,.
\]
Since in $\erre^n$ all norms are equivalent and the restriction $B_n:H_n\to H_n$ 
is continuous and bounded, by classical results based on Girsanov arguments, 
it follows that \eqref{eq_app} admits a weak solution 
$(\Omega, \cF, (\cF_t)_{t\geq0}, \P, X_n, W_n)$ with
\[
  X_n\in L^2(\Omega; C^0(\erre_+; H_n))\,.
\]
From \eqref{eq_app} it follows that 
\[
  X_n(t) = S(t)x_n + \int_0^tS(t-s)B_n(X_n(s))\,\d s + 
  \int_0^tS(t-s)G_n\,\d W_n(s)\,, \quad\forall\,t \geq0\,, \quad\P\text{-a.s.}
\]
Let us now show uniform estimates on $(X_n)_n$, in the two cases
(i) and (ii).

\smallskip
Case (i). First of all, for $\zeta\in[0,\xi)$,
using the regularisation properties of $S$
(see \cite[Prop.~2.2.9]{lunardi}), we have 
\begin{align*}
  &\norm{X_n(t)}_{2(\delta_0\wedge\zeta)} \\
  &\leq \norm{S(t)x_n}_{2(\delta_0\wedge\zeta)}
  +\int_0^t\norm{S(t-s)B_n(X_n(s))}_{2(\delta_0\wedge\zeta)}\,\d s
  +\norm{\int_0^tS(t-s)G_n\,\d W_n(s)}_{2(\delta_0\wedge\zeta)} \\
  &\leq \norm{x_n}_{2\delta_0}+\int_0^t
  \frac1{(t-s)^{(\delta_0\wedge\zeta)+\beta}}\norm{B_n(X_n(s))}_{-2\beta}\,\d s
  +\norm{\int_0^tS(t-s)G_n\,\d W_n(s)}_{2\zeta}\,.
\end{align*}
Hence, taking supremum in time and expectations, 
by \cite[Thm~5.15]{dapratozab} we deduce, for all $T>0$, that 
\begin{align*}
  &\E\sup_{t\in[0,T]}\norm{X_n(t)}_{2(\delta_0\wedge\zeta)} \\
  &\leq \norm{x}_{2\delta_0}+\sup_{y\in V_{2\alpha}}\norm{B(y)}_{-2\beta}
  \frac{T^{1-(\delta_0\wedge\zeta)-\beta}}{1-(\delta_0\wedge\zeta)-\beta}
  +C\E\left(\int_0^Tt^{-2\xi}\|S(t)G\|^2_{\cL^2(H,H)}\,\d t\right)^{1/2}\,,
\end{align*}
where we have used again that $S$ contracts in $V_{2\delta_0}$ and 
$V_{2\zeta}$, that 
$(\delta_0\wedge\zeta)+\beta\in[0,1)$, and the analyticity of the semigroup $S$.
This shows that there exists a constant $C_T>0$ independent of $n$ such that 
\beq
  \label{est1}
 \norm{X_n}_{L^2(\Omega; C^0([0,T]; V_{2(\delta_0\wedge\zeta)}))}\leq C_T\,, \quad\forall\,T>0\,.	
\eeq
Furthermore, deterministic maximal $L^2$-regularity 
(see \cite[Thms.~3.3]{vNVW}) and the 
analyticity of the semigroup $S$
(see \cite[Thm~5.15]{dapratozab}) ensure also that 
\[
  \norm{\int_0^\cdot S(\cdot-s)B_n(X_n(s))\,\d s}_{
  L^2(\Omega; L^2(0,T; V_{2-2\beta}))} \leq C_T \,,
  \quad\forall\,T>0
\]
and
\[
  \norm{\int_0^\cdot S(\cdot-s)G_n\,\d W_n(s)}_{
  L^2(\Omega; C^0([0,T]; V_{2\zeta}))} \leq C_T \,, 
  \quad\forall\,T>0\,.
\]
Moreover, since $\alpha<\xi$, we can fix $\zeta\in(\alpha,\xi)$:
with such a choice, let $\bar s \in(\alpha, \zeta)$ and note that 
by the regularising properties of $S$
one has
\[
  \|S(t)x_n\|_{2\bar s}\leq C\frac1{t^{\bar s-\delta_0}}\|x\|_{2\delta_0} \quad\forall\,t>0\,.
\]
Recalling that $\zeta< 1-\beta$ by assumptions \ref{H2}-\ref{H3},
as well as $\bar s-\delta_0<1$,
we infer that there exists $q>1$ such that 
\beq
  \label{est2}
 \norm{X_n}_{L^2(\Omega; L^q(0,T; V_{2\bar s}))}\leq C_T\,, \quad\forall\,T>0\,,
\eeq
In particular, since the embedding $V_{2\bar s}\embed V_{2\alpha}$ is compact, 
it follows that the laws of $(X_n)_n$ are tight on the space
\[
  L^q(0,T; V_{2\alpha})\cap C^0([0,T]; V_{2(\delta_0\wedge\zeta)})\,.
\]
Exploiting the fact that $B:V_{2\alpha}\to Z$ is strongly-weakly continuous, 
by letting $n\to\infty$ in \eqref{eq_app} yields the existence of a weak 
solution to \eqref{eq0} by classical stochastic compactness techniques. 

\smallskip
Case (ii). The proof is similar to the previous case, the only difference being 
the treatment of the stochastic convolution.
Indeed, one can estimate the terms directly 
in the space $D(A^{\delta_0\wedge\delta'})$ instead of $D(A^{\delta_0\wedge\zeta})$ as
\begin{align*}
  &\E\sup_{t\in[0,T]}\norm{X_n(t)}_{2(\delta_0\wedge\delta')} \\
  &\leq \norm{x}_{2\delta_0}+\sup_{y\in V_{2\alpha}}\norm{B(y)}_{-2\beta}
  \frac{T^{1-(\delta_0\wedge\delta')-\beta}}{1-(\delta_0\wedge\delta')-\beta}
  +\E\left(\int_0^T\norm{G}^2_{\cL^2(H, V_{2\delta'})}\,\d s\right)^{1/2}\,,
\end{align*}
where we have used that $S$ contracts in $V_{2\delta_0}$ and 
$V_{2\delta'}$, maximal regularity, 
the Burkholder-Davis-Gundy inequality, and that 
$(\delta_0\wedge\delta')+\beta\in[0,1)$.
Moreover, by 
\cite[Thm.~6.12, Prop.~6.18]{dapratozab}) one has also that
\[
  \norm{\int_0^\cdot S(\cdot-s)G_n\,\d W_n(s)}_{
  L^2(\Omega; L^2(0,T; V_{2\delta'+1}))} \leq C_T \,, 
  \quad\forall\,T>0\,.
\]
Moreover, since $\alpha<(\delta_0+1)\wedge(\delta'+\frac12)$, we can fix 
$\bar s \in(\alpha, (\delta_0+1)\wedge(\delta'+\frac12))$ and note that 
by the regularising properties of $S$
one has
\[
  \|S(t)x_n\|_{2\bar s}\leq C\frac1{t^{\bar s-\delta_0}}\|x\|_{2\delta_0} \quad\forall\,t>0\,.
\]
Recalling that $2\bar s<2\delta'+1< 2-2\beta$
by assumption,
as well as $\bar s-\delta_0<1$,
we infer that there exists $q>1$ such that 
\beq
  \label{est2_bis}
 \norm{X_n}_{L^2(\Omega; L^q(0,T; V_{2\bar s}))}\leq C_T\,, \quad\forall\,T>0\,,
\eeq
In particular, since the embedding $V_{2\bar s}\embed V_{2\alpha}$ is compact, 
it follows that the laws of $(X_n)_n$ are tight on the space
\[
  L^q(0,T; V_{2\alpha})\cap C^0([0,T]; V_{2(\delta_0\wedge\delta')})\,.
\]
The proof follows then as in the previous case.
\end{proof}

\begin{remark}
If $B$ is not bounded, one can obtain existence of local solutions
by classical stopping arguments. Instead, existence of global solutions 
in the unbounded case requires appropriate conditions on $B$,
according to the specific equation in consideration.
Typically, these can be assumptions of monotonicity type or 
assumption on the growth: for example, one can suppose that
there exists a constant $C_B>0$ such that
  \[
  \|B(x)\|_{-2\beta}\leq C_B\left(1 + \|x\|_{2\alpha}\right) \quad\forall\,x\in D(A^\alpha)\,.
  \]
In this framework,
the existence proof of Case (i) can be adapted with the additional condition 
that $\alpha\leq\delta_0$: indeed, one can estimate
\begin{align*}
  \int_0^t\frac1{(t-s)^{(\delta_0\wedge\zeta)+\beta}}\norm{B_n(X_n(s))}_{-2\beta}\,\d s
  \leq C_B\int_0^t\frac1{(t-s)^{(\delta_0\wedge\zeta)+\beta}}
  \left(1+\norm{X_n(s)}_{2\alpha}\right)\,\d s
\end{align*}
and conclude by using the Gronwall's 
lemma together with the relation $\alpha\leq\delta_0\wedge\zeta$.
Analogously, the existence proof of Case (ii)
can be adapted with the additional condition 
that $\alpha\leq\delta_0\wedge\delta'$:
indeed, one can write
\begin{align*}
  \int_0^t\frac1{(t-s)^{(\delta_0\wedge\delta')+\beta}}\norm{B_n(X_n(s))}_{-2\beta}\,\d s
  \leq C_B\int_0^t\frac1{(t-s)^{(\delta_0\wedge\delta')+\beta}}
  \left(1+\norm{X_n(s)}_{2\alpha}\right)\,\d s
\end{align*}
and conclude by using the Gronwall's 
lemma together with the relation $\alpha\leq\delta_0\wedge\delta'$.
\end{remark}
\endappendix

\section*{Acknowledgement}
F.B.\ is grateful to the Royal Society for financial support through Prof.~M.~Hairer's research professorship grant RP\textbackslash R1\textbackslash 191065. 
C.O. and L.S. are members of Gruppo Nazionale 
per l'Analisi Matematica, la Probabilit\`a 
e le loro Applicazioni (GNAMPA), 
Istituto Nazionale di Alta Matematica (INdAM), 
and gratefully acknowledge financial support 
through the project CUP\_E55F22000270001.
The present research is part of the activities of 
``Dipartimento di Eccellenza 2023-2027'' of Politecnico di Milano.

\small

\def\cprime{$'$}


\end{document}